\documentclass[a4paper, 11pt]{article}
\usepackage{amssymb}

\usepackage{amsmath,amssymb,amsthm}
\usepackage[latin1]{inputenc}
\usepackage{version,tabularx,multicol}
\usepackage{graphicx,float,psfrag}
\usepackage{stmaryrd}
 \usepackage{color}
\usepackage[pdftex]{hyperref}
\usepackage{amsfonts}
\usepackage{mathrsfs}
\usepackage{geometry}
\usepackage{cite}
\usepackage[numbers,sort&compress]{natbib}
\usepackage{xcolor}

\allowdisplaybreaks[4]

\theoremstyle{plain}
\newtheorem{theorem}{Theorem}[section]

 \newtheorem{lemma}[theorem]{Lemma}

\newtheorem{remark}{Remark}[section]

\newtheorem{assumption}{Assumption}[section]

\newcommand{\mE}{{\mathbb E}}

\newcommand{\mP}{{\mathbb P}}

\newcommand{\mR}{{\mathbb R}}

\newcommand{\ind}{{\bf 1}}

  \makeatletter
  \@addtoreset{equation}{section}
  \makeatother

 \def\beqlb{\begin{eqnarray}}\def\eeqlb{\end{eqnarray}}
 \def\beqnn{\begin{eqnarray*}}\def\eeqnn{\end{eqnarray*}}

 \def\mbb{\mathbb}

\newcommand{\bcen}{\begin{center}}
\newcommand{\ecen}{\end{center}}
\newcommand{\bgeqn}{\begin{equation}}
\newcommand{\edeqn}{\end{equation}}

\def\ez{\varepsilon}

\def\mE{{\mbb  E}}

\def\mP{{\mbb  P}}
\def\mR{{\mbb  R}}

\begin{document}
\title{Upper deviation probabilities for level sets of a supercritical branching random walk}
\author{Shuxiong Zhang, Lianghui Luo\footnote{E-mail: lianghui.luo@qq.com}}
\maketitle


\noindent\textit{Abstract:}
Given a supercritical branching random walk $\{Z_n\}_{n\geq0}$ on $\mathbb{R}$, let $Z_n([y,\infty))$ be the number of particles located in $[y,\infty)\subset \mathbb{R}$ at generation $n$. Let $m$ be the mean of the offspring law of $\{Z_n\}_{n\geq0}$ and $I(x)$ be the large deviation rate function of the underlying random walk of $\{Z_n\}_{n\geq0}$. It is known from \cite{Biggins1977} that under some mild conditions, for $x\in(0,x^*)$, $n^{-1}\log Z_n([ xn,\infty))$ converges almost surely to $\log m-I(x)$ on the event of nonextinction as $n\to\infty$, where $x^*$ is the speed of maximal position of the branching random walk. In this work, we investigate its upper deviation probabilities, in other words, the convergence rates of
$$\mathbb{P}\left(Z_n([xn,\infty))\geq e^{an}\right)$$
as $n\to\infty$, where $x>0$ and $a>(\log m-I(x))^+$. This paper is a counterpart work of the lower deviation probabilities \cite{zhanglowerlevel} and also completes those results in \cite{levelsetzhan} for the branching Brownian motion.

\bigskip
\noindent Mathematics Subject Classifications (2020): 60F10, 60J80, 60G50.

\bigskip
\noindent\textit{Keywords}: Branching random walk; Upper deviation probability; Level set.

\section{Introduction and Main results}
\subsection{Introduction}
Generally speaking, a branching random walk (BRW) is governed by a point process on $\mathbb{R}$. Every particle moves and splits
according to this point process (see \cite{zhan}). Nevertheless, in this article, we study a simplified version by assuming that the branch and the motion are independent. Given a probability distribution $\{p_k\}_{k\geq 0}$ on $\mathbb{N}$ and a real-valued random variable $X$, a branching random walk $\{Z_n\}_{n\geq 0}$ with offspring law $\{p_k\}_{k\geq 0}$ and step size $X$ is defined as follows. At time $0$, there is one particle located at the origin (i.e., $Z_0=\delta_0$). The particle dies and produces  offsprings according to the offspring distribution $\{p_k\}_{k\geq 0}$. Afterwards, the offspring particles move independently according to the law of $X$. This forms a point process at time $1$, denoted by $Z_1$. For any point process $Z_n$, $n\geq 2$, we define it by the following iteration
$$Z_n=\sum_{x\in Z_{n-1}}\tilde{Z}_1^{x},$$
where $\tilde{Z}_1^{x}$ has the same distribution as $Z_1(\cdot-S_x)$ and $\tilde{Z}_1^{x},\ x\in Z_{n-1}$ (conditioned on $Z_{n-1}$)  are independent. Here and later, for a point process $\xi$, $x\in\xi$ means $x$ is an atom of $\xi$, and $S_x$ is the position of $x$ (i.e., $\xi=\sum_{x\in\xi}\delta_{S_x}$).\par
For $A\subset\mathbb{R}$, let
$$Z_n(A):=\#\left\{u\in Z_n: S_u\in A\right\},$$
i.e. the number of particles located in the set $A$. Write $|Z_n|:=Z_n(\mathbb{R})$, $n\geq0$ and $m:=\mathbb{E}[|Z_1|]$.
\par

Usually, we call $\{|Z_n|\}_{n\geq0}$ the branching process or the Galton-Watson process. $\{|Z_n|\}_{n\geq0}$ is called a supercritical (critical, subcritical) branching process if $m>1$ ($=1$, $<1$). In supercritical case, the branching process will survive with positive probability, otherwise it will die out almost surely (provided $p_1\neq1$). For a more detailed discussion, one can refer to \cite{Athreya}. In this paper, we always deal with the supercritical case.

\par
 According to Biggins \cite[Theorem 2]{Biggins1977}, if $1<m<\infty$, $\mathbb{E}[X]=0$ and $\mathbb{E}[e^{\kappa X}]<\infty$ for some $\kappa>0$, then for $x\in[0,x^*)$,
\begin{align}\label{4re32}
\lim_{n\rightarrow\infty}\frac{1}{n}\log Z_n([xn,\infty))=\log m-I(x),~\mathbb{P}\text{-a.s.~on nonextinction},
\end{align}
where $I(x):=\sup_{t\geq 0}\{tx-\log \mathbb{E}[e^{tX}]\}$ is the rate function (see \cite{Biggins1977}) and
\[x^*:=\sup\{y\geq0:I(y)\leq\log m\}\]
is the speed of maximal position of a branching random walk (see \cite{Hammersley74}). In this work, we are going to study the corresponding upper deviation probabilities, i.e. the decay rate of
$$\mathbb{P}(Z_n([xn,\infty))\geq e^{an})$$
as $n\rightarrow\infty$, where $x>0$ and $a>(\log m-I(x))^+$.\par
In fact, this question was first studied by A\"idekon, Hu and Shi \cite[Theorem 1.1]{levelsetzhan} for the binary branching Brownian motion $\{Z_t\}_{t\geq0}$. To be specific, they proved that if $x>0$ and $(1-\frac{x^2}{2})^+<a<1$, then
$$\lim_{t\to\infty}\frac{1}{t}\log\mP(Z_t([xt,\infty))\geq e^{at})=-\left[\frac{x^2}{2(1-a)}-1\right].$$
\par
 A branching random walk can be viewed as a discrete but somewhat generalized version of the classical branching Brownian motion. Thus, our work is a generalization of \cite[Theorem 1.1]{levelsetzhan}. However, since our assumptions about the offspring law and step size are much weaker than \cite{levelsetzhan}, this makes the proof delicate. For example, in our setting, the rate function $I(x)$ and the distribution of $|Z_n|$ do not have explicit expression, some equations and deviation probabilities involved can not be dealt with as those in \cite{levelsetzhan}. The difficulty related to $I(x)$ is overcome by a careful analysis of several optimization problems using the convexity and asymptotics of $I(x)$. For $|Z_n|$, we use the uniform boundedness of the exponential moment of $|Z_n|m^{-n}$ (see \cite[Theorem 4]{athreya94}) to obtain upper deviation probabilities of it. Furthermore, we find the following new decay scales. If the essential supremum of the step size is bounded, then $\mathbb{P}(Z_n([xn,\infty))\geq e^{an})$ may decay double-exponentially. We also consider the case $a\geq\log m$, which has not been studied in \cite{levelsetzhan}, and the decay scale also turns out to be double-exponential. Moreover, we show that heavy-tailed offspring law may lead phase transition for the upper deviation probability.

\par
For the corresponding lower deviation probabilities of level sets, namely the decay rate of
$\mathbb{P}(Z_n([xn,\infty))<e^{an}),$
where $x\in(0,x^*)$ and $0\leq a<\log m-I(x)$, one can refer to Zhang \cite{zhanglowerlevel}; see also \"{O}z \cite{Mehmet} for the lower deviation probability of local mass of a branching Brownian motion. The lower deviation probability and upper deviation probability are essentially two different questions. Firstly, the lower deviation aims at keeping a small population in $[xn,\infty)$ at time $n$. So, the strategy is to force particles before some intermediate time $t_n$ to produce as less descendants as possible, and in the meanwhile, one should control individuals at time $t_n$ locate at a low position. However, for the upper deviation, the strategy is to force the maximum of the branching random walk at time $t_n$ become large enough such that the sub-branching random walk emanating from the maximum at time $t_n$ can normally produce $e^{an}$ descendants in $(nx,\infty)$ at time $n$.
Secondly, in the lower deviation probability, the decay scales mainly depend on the tail probability of the step size and whether $p_1$ is 0 or not. But, in the upper deviation probability, the decay scales depend on the tail probability of the offspring law and the size of $a$ and $\log m$.
Therefore, both in terms of method and conclusion, to study the the upper deviation will make some progress of the large deviation theory of the branching random walk.

\par
 We mention here that, since the last few decades, the branching random walk has been extensively studied due to its connection to many fields, such as Gaussian multiplicative chaos, random walk in random environment, random polymer, random algorithms and discrete Gaussian free field etc; see \cite{hushi}, \cite{liu06}, \cite{bramsonding15} and \cite{levelsetzhan} references therein. One can refer to Shi \cite{zhan} for a more detailed overview. Especially, the large deviation probability (LDP) for branching random walks and branching Brownian motions on real line have attracted many researchers' attention. For example, Hu \cite{yhu}, Gantert and H{\"o}felsauer \cite{Gantert18} and Chen and He \cite{Helower} considered the LDP and the moderate deviation probability of the branching random walk's maximum (for the maximum of branching Brownian motion, see Chauvin and Rouault \cite{chauvin} and Derrida and Shi \cite{derrida16,derrida17,derrida172}). For the LDP of empirical distribution, see Louidor and Perkins \cite{Louidor}, Chen and He \cite{ChenHe} and Zhang \cite{zhangempirical}. Some other related works include Rouault \cite{rouault}, Buraczewski and Ma\'slanka \cite{Burma} and Bhattacharya \cite{bhattacharya}.
 \subsection{Main Results}

 Let $L\in(-\infty,\infty]$ be the essential supremum of the step size $X$, i.e. $L=\sup\{x\in \mR:\mathbb{P}(X>x)>0\}$. In the
sequel of this work, we always need the following assumptions.
\begin{assumption}\label{t335rw2}\leavevmode \\
(i) $p_0=0$, $1<m<\infty$;\\
(ii) $\mathbb{E}[X]=0$, $\mP(X=0)<1$ and $\mathbb{E}[e^{\kappa X}]<\infty$ for some $\kappa>0$.
\end{assumption}
 \begin{remark} If $p_0>0$, using similar arguments of this paper, one can also obtain the following theorems under the probability $\mathbb{P}(\ \cdot\ |\ |Z_n|>0,\ \forall\ n\geq 0)$. $\mathbb{E}[X]=0$ is only made to simplify the statement. $\mP(X=0)<1$ is made to avoid trivial case. The other assumptions are necessary for the almost sure convergence (\ref{4re32}).
 \end{remark}
 Now, we are ready to show our main results. Recall that $I(x)=\sup_{t\geq 0}\{tx-\log \mathbb{E}[e^{tX}]\}$.
\begin{theorem}\label{upth1}
Assume $\mathbb{E}[e^{\theta |Z_1|}]<\infty$ for some $\theta>0$, $x\in(0,L)$ and $a\in((\log m-I(x))^+,\log m)$. If
 $L=\infty$, then
$$\lim\limits_{n\rightarrow\infty}\frac{1}{n}\log\mathbb{P}(Z_n([xn,\infty))\geq e^{an})=-I(a,x)\in(-\infty,0),$$
where
$$I(a,x):=\inf_{s\in\left(0,1-\frac{a}{\log m}\right],y\in(0,x]:\atop
~\log m-I\left(\frac{x-y}{1-s}\right)=\frac{a}{1-s}}\left\{sI\left(\frac{y}{s}\right)-s\log m\right\}.$$
\end{theorem}
\begin{remark} According to the local survival probabilities studied in \cite{rouault} and truncation arguments used in \cite{bertoin}, one can obtain the upper deviation probabilities of the maximum of a general branching random walk (i.e. defined by a point process). Then, making use of the upper deviation probabilities of the maximum, one can generalize Theorem \ref{upth1} to the general branching random walk setting.
\end{remark}
\begin{remark}\label{4r4rdef}
Denote by $y_s$ the unique solution of the equation (w.r.t. $y$)
\begin{equation}\label{solution}
 \log m-I\left(\frac{x-y}{1-s}\right)=\frac{a}{1-s}
\end{equation}
 on $(0,x]$, where $s\in(0,1-\frac{a}{\log m}]$; see Lemma \ref{yohua} for the proof. Theorem \ref{upth1} still holds if we replace the condition $L=\infty$ with one of the following conditions:\\
 (i) $L<\infty$, $x^*<L$, $\mathbb{P}(X=L)=0$ and there exists some $s\in\big(0,1-\frac{a}{\log m}\big]$ such that
    $$y_s/s< L;$$
 (ii) $L<\infty$, $x^*<L$, $\mathbb{P}(X=L)>0$ and there exists some $s\in\big(0,1-\frac{a}{\log m}\big]$ such that $$y_s/s\leq L.$$
\end{remark}
\begin{remark}We give several examples to illustrate Remark \ref{4r4rdef}. \\
(i) If $X$ follows uniform distribution on $(-1,1)$, $a\in\left((\log m-I(x))^+,(1-\frac{x}{L})\log m\right)$ and $s^*:=1-\frac{a}{\log m}$, then $y_{s^*}=x$. This implies $X$ satisfies Remark \ref{4r4rdef} (i).  \\
(ii) Let $X$ satisfy $\mathbb{P}(X=-1)=\mathbb{P}(X=1)=0.5$, $a\in\left((\log m-I(x))^+,(1-\frac{x}{L})\log m\right]$ and $m=1.5$. Then, similarly to (i), it is simple to see that $X$ satisfies Remark \ref{4r4rdef} (ii). \\
(iii) Now, we give an example such that none of the Remark \ref{4r4rdef} (i) and (ii) holds. Since $\lim_{s\to0}\frac{y_s}{s}=\infty$ (see (\ref{47ytgth12cdgr}) below), there exists $\delta\in(0,1)$ such that for any $s\in(0,\delta)$, $\frac{y_s}{s}>2L$ and $(1-\delta)\log m>\log m-I(x)$. Finally, one can see that if $a\in((1-\delta)\log m, \log m)$, then $\frac{y_s}{s}>2L$ for $s\in\big(0,1-\frac{a}{\log m}\big]$.

\end{remark}

In the next theorem, we will show that if none of Remark \ref{4r4rdef} (i) and (ii) holds, then $\mP(Z_n([xn,\infty))\geq e^{an})$ may decay double-exponentially. Put $b:=\inf\{k\geq m:p_k>0\}$.
\begin{theorem}\label{12defedi} Assume $\mathbb{E}[e^{\theta |Z_1|}]<\infty$ for some $\theta>0$, $x\in(0,L)$ and $a\in((\log m-I(x))^+,\log m)$.  Suppose $L<\infty$ and one of the following two conditions holds.\\
\noindent(i) $x^*=L;$\\
\noindent(ii)$x^*<L$ and
    $$y_s/s> L, \forall s\in\Big(0,1-\frac{a}{\log m}\Big].$$
Then
\begin{align}
0&<\liminf_{n\to\infty}\frac{1}{n}\log\left[-\log\mathbb{P}(Z_n([xn,\infty))\geq e^{an})\right]\cr
&\leq\limsup_{n\to\infty}\frac{1}{n}\log\left[-\log\mathbb{P}(Z_n([xn,\infty))\geq e^{an})\right]\leq c^*\log b,\nonumber
\end{align}
where $c^*\in(0,\frac{x}{L})$ is the unique solution of the equation (w.r.t. $c$)
$$
\log m-I\left(L+\frac{x-L}{1-c}\right)-\frac{a-\log b}{1-c}=\log b.
$$
\end{theorem}

Though we get Remark \ref{4r4rdef} and Theorem \ref{12defedi}, we are unable to obatin the decay scale of $\mP(Z_n([xn,\infty))\geq e^{an})$ if
 $L<\infty$, $\mP(X=L)=0$ and
\[ \inf_{s\in(0,1-\frac{a}{\log m}]}\frac{y_s}{s}=L.\]

 In the next, we consider the case $a\geq \log m$. Let $R_n$ be the rightmost position of a branching random walk with offspring law $p_k$ and step size $-X$ at generation $n$. If there exists some $k\geq2$ such that $p_k=1$, then it is easy to see that
\begin{align}
\mathbb{P}(Z_n([xn,\infty))\geq e^{an})=
\begin{cases}
\mathbb{P}(R_n\leq -xn ),&~a=\log m;\cr
0,&~a>\log m,\nonumber
\end{cases}
\end{align}
where $\mP(R_n\leq -xn)$ has been well studied in \cite{ChenHe}. Thus, we only need to consider the case that the branch is random. Define $\mu:=\sup\{k:p_k>0,~k\geq1\}\in(0,\infty]$ and $\alpha:=\inf\{k:p_{k}>0,k\geq e^a\}$.
\begin{theorem}\label{remark2}
 Assume $\mathbb{E}[e^{\theta |Z_1|}]<\infty$ for some $\theta>0$. Suppose $p_k<1$ for any $k\geq2$, $x\in(0,L)$ and $a\in[\log m, \log\mu]\cap\mathbb{R}$. Then

\begin{align}\label{4rrfefcb}
a-\log m&\leq\liminf_{n\to\infty}\frac{1}{n}\log[-\log\mP(Z_n([xn,\infty))\geq e^{an})]\cr
&\leq \limsup_{n\to\infty}\frac{1}{n}\log[-\log\mP(Z_n([xn,\infty))\geq e^{an})]\leq\log \alpha.
\end{align}
\end{theorem}
\begin{remark}\label{remark3}Although the lower bound of (\ref{4rrfefcb}) is degenerate in the case $a=\log m$, one can use Theorem \ref{12defedi} to show that the lower bound is still positive if $L<\infty$; see Section \ref{sec3} for the proof.
\end{remark}
\begin{remark}\label{remark1} Theorem \ref{upth1} still holds, if we replace the condition $\mathbb{E}[e^{\theta |Z_1|}]<\infty$ for some $\theta>0$ with $\lim_{x\to\infty}\mathbb{P}(|Z_1|>x) e^{-\lambda x^{\alpha}}=1$ for some $\lambda>0$, $\alpha\in(0,1)$; see Subsection \ref{subsec3} for the proof.
\end{remark}

Though Remark \ref{remark1} tells us that the assumption $\mathbb{E}[e^{\theta |Z_1|}]<\infty$ can be weaken, the following theorem shows that if the offspring law has a Pareto tail, then the limit will change. To demonstrate this, we consider the standard normal step size. Hence, $I(x)=x^2/2$ for $x\geq0$. We write $f(y)=\Theta(1)g(y)$ if $0<\liminf_{y\to+\infty} f(y)/g(y)\leq \limsup_{y\to+\infty} f(y)/g(y)<\infty$.
\begin{theorem}\label{upth2} Assume~$\mathbb{P}(|Z_1|>y)=\Theta(1)y^{-\beta}$ for some $\beta>1$. Suppose $X$ is a standard normal random variable,~$x>0$ and~$a\in((\log m-\frac{x^2}{2})^+,\infty)$. Then
\begin{align}
-\left[(\beta-1)a+a-\left(\log m-\frac{x^2}{2}\right)\right]&\leq\liminf_{n\rightarrow\infty}\frac{1}{n}\log\mathbb{P}(Z_n([xn,\infty))\geq e^{an})\cr
&\leq\limsup_{n\rightarrow\infty}\frac{1}{n}\log\mathbb{P}(Z_n([xn,\infty))\geq e^{an})\leq-\left[a-\left(\log m-\frac{x^2}{2}\right)\right].\nonumber
\end{align}
\end{theorem}
\begin{remark} In Theorem \ref{upth1}, if $X$ is a standard normal random variable, then
$$
I(a,x)=\frac{x^2\log m}{2(\log m-a)}-\log m.
$$
Thus, if $a\in(\log m-\frac{x^2}{2\beta},\log m)$, then the lower bound in Theorem \ref{upth2} is larger than the limit of Theorem \ref{upth1}.
\end{remark}
The rest of this paper is organised as follows. In Section \ref{sec1}, we consider the case $a\in((\log m-I(x))^+,\log m)$ and
 $L=\infty$, where Theorem \ref{upth1} is proved. This section is divided into three subsections. In the first subsection, we present several lemmas related to rate functions $I(x)$ and $I(a,x)$. In the second and third subsection, we prove the lower bound and upper bound of Theorem \ref{upth1}, respectively. Section \ref{sec2} is devoted to study the case $a\in((\log m-I(x))^+,\log m)$ and
 $L<\infty$, where Theorem \ref{12defedi} is proved. We deal with the case $a\geq\log m$ in Section \ref{sec3}, and Theorem \ref{remark2} is proved. We prove Theorem \ref{upth2} in Section \ref{sec4}, where the offspring law is assumed to be a Pareto tail.
\section{Proof of Theorem \ref{upth1}}\label{sec1}
In this section, we are going to prove Theorem \ref{upth1}, and the proofs also carry through for Remark \ref{4r4rdef}. This section is divided into three subsections. In subsection \ref{subsec1}, we present several lemmas, which concern the asymptotic behaviour of level sets and properties of the rate function $I(a,x)$. In subsection \ref{subsec2}, we use the upper deviation probability of the maximum for the branching random walk (see \cite{Gantert18}) to obtain the lower bound of Theorem \ref{upth1}. We prove the upper bound of Theorem \ref{upth1} in subsection \ref{subsec3}. The proof mainly uses the path partition technique in \cite{levelsetzhan}. But since we consider more general step size and offspring law, the details are much more involved.
\subsection{Preliminaries}\label{subsec1}
In the remainder of the paper, for ease of notation, we write $Z_n A=Z_n(A)$ for a Borel set $A\subset \mathbb{R}$. According to  \cite[Theorem 2]{Biggins1977}, we have the following lemma.
\begin{lemma}\label{a.s.} The following holds, $\mathbb{P}$ almost surely:
\begin{equation}
\lim\limits_{n\rightarrow\infty}\frac{\log Z_n[xn,\infty)}{n}=
\left\{
\begin{array}{rcl}
~~~~0~~~~,           & & {x> x^*};\\
\log m-I(x), & & {0<x<x^*};\\
~~~~\log m~~~~,       & & {x\leq 0}.
\end{array} \right.
\end{equation}
\end{lemma}
The next lemma provides several properties of the rate function $I(x)=\sup_{t\geq 0}\{tx-\log \mathbb{E}[e^{tX}]\}$, which are borrowed from \cite[Lemma 2.3]{zhanglowerlevel}.
Let $\lambda^*:=\sup\left\{\lambda\geq0:\mathbb{E}[e^{\lambda X}]<\infty\right\}$ and $\Lambda(\lambda):=\log \mathbb{E}\left[e^{\lambda X}\right]$. Note that by Assumption \ref{t335rw2} (i), we have $\lambda^*>0$. Furthermore, by \cite[Exercise 2.2.24]{Dembo}, $\Lambda(\lambda)$ is smooth and convex on $(0,\lambda^*)$.
\begin{lemma}\label{sratefunction} The rate function $I(x)$ can be classified into the following three cases.\\
(i) If $\Lambda'(\lambda)\uparrow+\infty$ as $\lambda\rightarrow \lambda^*$, then for any $x>0$, there exists some $\lambda\in(0,\lambda^*)$ such that
$$x=\Lambda'(\lambda),~I(x)=\lambda\Lambda'(\lambda)-\Lambda(\lambda).$$~~~~Furthermore, $\lim_{x\rightarrow+\infty} \frac{I(x)}{x}=\lambda^*$.\\
(ii) If $\lambda^*=+\infty$, $\Lambda'(\lambda)$ converges to some finite limit as $\lambda\rightarrow\infty$, then $\Lambda'(\lambda)\uparrow L$ and

$$I(x)=
\begin{cases}
\text{positive finite}, &x\in(0,L);\cr
-\log\mathbb{P}(X=L), &x=L;\cr
+\infty, &x\in (L,\infty).
\end{cases}
$$
~~~~Furthermore, as $x\uparrow L$, $I(x)\uparrow -\log\mathbb{P}(X=L).$\\
(iii) If $0<\lambda^*<+\infty$, $\Lambda'(\lambda)$ converges to some finite limit $T$ as $\lambda\rightarrow \lambda^*$, then $\mathbb{E}\left[e^{\lambda^*X}\right]<\infty$ and
$$I(x)=
\begin{cases}
\text{positive finite}, &x\in(0,T];\cr
\lambda^*x-\log \mathbb{E}\left[e^{\lambda^*X}\right], &x\in [T,\infty).
\end{cases}
$$
\end{lemma}
\begin{remark}For instance, the following three examples satisfy Lemma \ref{sratefunction} (i), (ii) and (iii), respectively:
(i)\ If $X$ is a standard normal random variable, then $\Lambda(\lambda)=\frac{\lambda^2}{2}$ and $\lambda^*=\infty$;
(ii)\ If~$\mathbb{P}(X=1)=\mathbb{P}(X=-1)=0.5$, then $\Lambda(\lambda)=\log(e^{\lambda}+e^{-\lambda})-\log2$ and $\lim_{\lambda\to\infty}\Lambda'(\lambda)=1$;
(iii) \ Let $Y$ be a random variable with density function
\[
f(x)=
\begin{cases}
\frac{Ce^{-x}}{x^3}, &x\geq 1;\cr
 0, &x<1,
\end{cases}
\]
where $C$ is chosen such that $\int_{\mathbb{R}} f(x)\mathrm{d} x=1$. Then we can define $X:=Y-\mathbb{E} [Y]$, and obviously $\lambda^*=1$.
\end{remark}
The following lemma can be inferred from \cite[Exercise 2.2.24]{Dembo}.
\begin{lemma} \label{3ew5tge}
$I(x)$ is infinitely differentiable and strictly increasing on $[0,L)$.
\end{lemma}

In the proof of lower and upper bound of Theorem \ref{upth1}, we will encounter different optimization problems. The following lemma shows the equivalency of these optimization problems.

\begin{lemma}\label{yohua}
Assume $x^*<L$, $x\in(0,L)$ and $a\in((\log m-I(x))^+,\log m)$.
 Suppose that one of the following three conditions holds:\\
 (i) $L=\infty$;\\
(ii) $L<\infty$, $\mathbb{P}(X=L)=0$ and there exists some $s\in\left(0,1-\frac{a}{\log m}\right]$ such that
    $$\frac{y_s}{s}< L;$$
 (iii) $L<\infty$, $\mathbb{P}(X=L)>0$ and there exists some $s\in\left(0,1-\frac{a}{\log m}\right]$ such that $$\frac{y_s}{s}\leq L.$$
 Then,
\begin{align}\label{4tfr3ws}
&\inf_{s\in(0,1),y\in(0,x):\atop~\frac{y}{s}>x^*,~\log m-I\left(\frac{x-y}{1-s}\right)\geq\frac{a}{1-s}}\left\{sI\left(\frac{y}{s}\right)-s\log m\right\}\cr
&=\inf_{s\in(0,1), y\in(0,x):\atop
~\log m-I\left(\frac{x-y}{1-s}\right)\geq\frac{a}{1-s}}\left\{sI\left(\frac{y}{s}\right)-s\log m\right\}\cr
&=\inf_{s\in\left(0,1-\frac{a}{\log m}\right],y\in(0,x]:\atop
~\log m-I\left(\frac{x-y}{1-s}\right)=\frac{a}{1-s}}\left\{sI\left(\frac{y}{s}\right)-s\log m\right\}=I(a,x)\in(0,\infty).
\end{align}
Furthermore, $\lim_{a'\to a-}I(a',x)=I(a,x)$.
\end{lemma}
\begin{proof} We first prove the second equality in (\ref{4tfr3ws}). Since $I(\cdot)$ is continuous and strictly increasing on $[0,L)$ (see Lemma \ref{3ew5tge}), we have the following:\\
(a) For~$s>1-\frac{a}{\log m}$, since~$\log m-\frac{a}{1-s}<0$, the inequality (w.r.t. $y$)
$$
\log m-I\left(\frac{x-y}{1-s}\right)\geq\frac{a}{1-s}
$$
has no solution on $(0,x]$;\\
(b) For $s\in[0,1-\frac{a}{\log m}]$, the assumption $a>(\log m-I(x))^+$ and the inequality
$$\log m-I\left(\frac{x-y}{1-s}\right)\geq \frac{a}{1-s}$$
imply that

$$0\leq I\left(\frac{x-y}{1-s}\right)\leq \log m-\frac{a}{1-s}\leq \log m-a\leq I(x).$$
Since $I(x)$ is strictly increasing on $[0,x]$, there exists a unique $\bar{x}\in[0,x]$ such that $I(\bar{x})=I\left(\frac{x-y}{1-s}\right)$. This entails that
$$
\frac{x-y}{1-s}=\bar{x}.
$$
So, $y_s=x-(1-s)\bar{x}$ is the unique solution of (\ref{solution}) on $(0,x]$.
\par

Combining (a), (b) and the fact that
$f(s,y):=sI\left(\frac{y}{s}\right)-s\log m$ is increasing w.r.t. $y$, we thus have
\begin{align}\label{2344345ws}
\inf_{s\in(0,1),y\in(0,x):\atop~\log m-I\left(\frac{x-y}{1-s}\right)\geq\frac{a}{1-s}}\left\{sI\left(\frac{y}{s}\right)-s\log m\right\}
=\inf_{s\in(0,1-\frac{a}{\log m}],y\in(0,x]:\atop~\log m-I\left(\frac{x-y}{1-s}\right)=\frac{a}{1-s}}\left\{sI\left(\frac{y}{s}\right)-s\log m\right\}.
\end{align}
\par
We proceed to prove the first equality in (\ref{4tfr3ws}). It suffices to show that $\log m-I\left(\frac{x-y}{1-s}\right)\geq\frac{a}{1-s}$ implies
$$\frac{y}{s}>x^*.$$
In fact, since $a>\log m-I(x)$, we have
\begin{align}\label{5t4rg2w}
\log m-I\left(\frac{x-y}{1-s}\right)&\geq \frac{a}{1-s}>\frac{\log m-I(x)}{1-s}.
\end{align}
Furthermore, from the assumption $x^*<L$,
$$
I(x^*)=\log m.
$$
Thus, by (\ref{5t4rg2w}),
\begin{align}\label{4rewe22d}
I(x)&> s\log m+(1-s)I\left(\frac{x-y}{1-s}\right)\cr
&= sI(x^*)+(1-s)I\left(\frac{x-y}{1-s}\right)\cr
&\geq I(sx^*+x-y),
\end{align}
where the last inequality comes from the convexity of $I(\cdot)$.
Since $I(\cdot)$ is strictly increasing on $(0,L)$, (\ref{4rewe22d}) yields that
\begin{equation}\label{efrefe}
\frac{y}{s}> x^*.
\end{equation}
Therefore,
$$
\inf_{s\in(0,1),y\in(0,x):\atop~\frac{y}{s}>x^*,~\log m-I\left(\frac{x-y}{1-s}\right)\geq\frac{a}{1-s}}\left\{sI\left(\frac{y}{s}\right)-s\log m\right\}=\inf_{s\in(0,1),y\in(0,x):\atop~\log m-I\left(\frac{x-y}{1-s}\right)\geq\frac{a}{1-s}}\left\{sI\left(\frac{y}{s}\right)-s\log m\right\}.
$$
\par
We proceed to prove $I(a,x)\in(0,\infty)$. Note that
$$
I(a,x)=\inf_{s\in(0,1),y\in(0,x]:\atop~\log m-I\left(\frac{x-y}{1-s}\right)=\frac{a}{1-s}}\left\{sI\left(\frac{y}{s}\right)-s\log m\right\}=\inf_{s\in(0,1-\frac{a}{\log m}]}\left\{sI\left(\frac{y_s}{s}\right)-s\log m\right\}.
$$
It is simple to see $y_s$, $s\in(0,1-\frac{a}{\log m}]$ is a continuous function. Since $I(0)=0$, by the convexity of $I(\cdot)$,
$$sI\left(\frac{y_s}{s}\right)=(1-s)I(0)+sI\left(\frac{y_s}{s}\right)\geq I(y_s).$$
Hence, there exists some $\delta\in (0,1-\frac{a}{\log m}]$ such that for any $s\in(0,\delta]$,
\begin{align}\label{7hjbsf}
sI\left(\frac{y_s}{s}\right)-s\log m>\frac{I(y_0)}{2}>0.
\end{align}
\par
 If the condition (i) holds, then $I(x)$ is a continuous function on $\mathbb{R}$. So, $f(s,y_s)=sI\left(\frac{y_s}{s}\right)-s\log m$ is also continuous on the compact set
 $[\delta,1-\frac{a}{\log m}]$. This, combined with the fact that $y_s/s>x^*$, implies that
  \begin{align}\label{6bu}
\inf_{s\in[\delta,1-\frac{a}{\log m}]}\left\{sI\left(\frac{y_s}{s}\right)-s\log m\right\}>0.
\end{align}
Thus, together with (\ref{7hjbsf}), we conclude $I(a,x)>0$ in the case (i).
 \par
 If the condition (ii) holds, then by Lemma \ref{sratefunction} (ii),
 \begin{align}\label{4rrw34rf}
 \lim_{x\to L}I(x)=+\infty.
 \end{align}
Moreover, there exists some $\tilde{s}\in(0,1-\frac{a}{\log m}]$ such that
$$
q:=\frac{y_{\tilde{s}}}{\tilde{s}}<L.
$$
So, there exists $\delta'>0$ such that
 \begin{align}\label{6bu99hf}
\inf_{s\in A}\left\{sI\left(\frac{y_s}{s}\right)-s\log m\right\}>0,
\end{align}
where $A:=\left\{s\in[\delta,1-\frac{a}{\log m}]:\frac{y_s}{s}\geq L -\delta'\right\}$. Since $\delta, \delta'$ can be arbitrarily small, we choose $\delta<\tilde{s}$ and $\delta'<L-q$. This means that $A':=\left\{s\in[\delta,1-\frac{a}{\log m}]:\frac{y_s}{s}\leq L -\delta'\right\}$ is non-empty. Because $f(s,y_s)$ is continuous on the compact set $A'$ and $y_s/s>x^*$, we have
\begin{align}\label{6bu99}
\inf_{s\in A'}\left\{sI\left(\frac{y_s}{s}\right)-s\log m\right\}>0.
\end{align}
 (\ref{6bu}), together with (\ref{6bu99hf}) and (\ref{6bu99}), concludes $I(a,x)>0$ in the case (ii).
\par
If the condition (iii) holds, then by Lemma \ref{sratefunction} (ii), $I(x)$ is continuous on $(0,L]$ and $I(x)=+\infty$ on $(L,\infty)$. Thus, $f(s,y_s)$ is continuous on the compact set
$$\left\{s\in[\delta,1-\frac{a}{\log m}]:\frac{y_s}{s}\leq L \right\}.$$
It follows that
\begin{align}\label{6u}
\inf_{s\in[\delta,1-\frac{a}{\log m}]}\left\{sI\left(\frac{y_s}{s}\right)-s\log m\right\}>0.
\end{align}
This, combined with (\ref{6bu}), concludes $I(a,x)>0$ in the case (iii). $I(a,x)<\infty$ follows directly from the definition.
\par
The proof of $\lim_{a'\to a-}I(a',x)=I(a,x)$ is similar to the proof of Lemma \ref{yohuaepsilon}, so we omit the proof here.
\end{proof}
The following lemma will be used in the proof of upper bound of Theorem \ref{upth1}.
\begin{lemma}\label{yohuaepsilon} Under the hypothesis of Lemma \ref{yohua}, we have
\[
\sup_{s\in(0,1),y> \varepsilon,z\geq (x-\varepsilon):\atop
 (1-s)\log m-I\left(\frac{z-y}{1-s}-\varepsilon\right)(1-s)\geq (a-\varepsilon)} \left\{s\log m-I\left(\frac{y}{s}-\varepsilon\right)s\right\}
 \rightarrow -I(a,x)~\text{as}~\varepsilon\to0+.
 \]
\end{lemma}
\begin{proof}
 Fix $\varepsilon\in(0,a-(\log m-I(x))^+)$. For $s\in(0,1-\frac{a-\varepsilon}{\log m}]$, since $0\leq\log m-\frac{a-\varepsilon}{1-s}\leq\log m-a+\varepsilon\leq I(x)$, the equation
\begin{align}\label{23r4redr}
(1-s)\log m-I\left(\frac{x-\varepsilon-y}{1-s}-\varepsilon\right)(1-s)=(a-\varepsilon)
\end{align}
has a unique solution on $[(x-\varepsilon)-(x+\varepsilon)(1-s),(x-\varepsilon)-(1-s)\varepsilon]$, denoted by $y(s,\varepsilon)$. Note that for $s>1-\frac{a-\varepsilon}{\log m}$, the inequality
$$(1-s)\log m-I\left(\frac{z-y}{1-s}-\varepsilon\right)(1-s)\geq (a-\varepsilon)$$
has no solution.  On the other hand, since $y(0,0)>0$ and $y(s,\varepsilon)$ is continuous, there exists some $\eta>0$ such that for $\varepsilon\in(0,\eta)$ and $s\in(0,1-\frac{a-\varepsilon}{\log m}]$,
\begin{equation}\label{eq3}
y(s,\varepsilon)\geq y(0,\varepsilon)>\frac{y(0,0)}{2}>\varepsilon,
\end{equation}
where the first inequality follows from the fact that $y(s,\varepsilon)$ is increasing w.r.t. $s$. Thus, for $\varepsilon\in(0,\eta)$,
\begin{align}\label{eq13}
&\sup_{s\in(0,1),y> \varepsilon,z\geq (x-\varepsilon):\atop
 (1-s)\log m-I\left(\frac{z-y}{1-s}-\varepsilon\right)(1-s)\geq (a-\varepsilon)} \left\{s\log m-I\left(\frac{y}{s}-\varepsilon\right)s\right\}\cr
 &=\sup_{s\in(0,1-\frac{a-\varepsilon}{\log m}],y> \varepsilon:\atop
 (1-s)\log m-I\left(\frac{x-\varepsilon-y}{1-s}-\varepsilon\right)(1-s)=(a-\varepsilon)} \left\{s\log m-I\left(\frac{y}{s}-\varepsilon\right)s\right\}\cr
 &=\sup_{s\in\left(0,1-\frac{a-\varepsilon}{\log m}\right]} \left\{s\log m-I\left(\frac{y(s,\varepsilon)}{s}-\varepsilon\right)s\right\}=:L_{\varepsilon}.
\end{align}
For convenience, write
$$L_0:=-I(a,x)=\sup_{s\in\left(0,1-\frac{a}{\log m}\right]}\left\{s\log m-I\left(\frac{y(s,0)}{s}\right)\right\}.$$
Since $y(s,\varepsilon)$ is decreasing w.r.t. $\varepsilon$, we have
\begin{equation}\label{dsfgdg}
\liminf_{\varepsilon\rightarrow0+}L_{\varepsilon}\geq L_0.
\end{equation}
\par
It suffices to show
$$\limsup_{\varepsilon\to 0+}L_\varepsilon\leq L_0.$$
Set
$$g(s,\varepsilon):=s\log m-sI\Big(\frac{y(s,\varepsilon)}{s}-\varepsilon\Big),~s\in\Big(0,1-\frac{a}{\log m}\Big],~\varepsilon\in[0,\eta).$$
From Lemma \ref{sratefunction}, there exists a constant $\kappa\in(0,+\infty]$ such that $\lim_{t\to\infty} I(t)/t=\kappa$. Since
\begin{align}\label{47ytgth12cdgr}
\lim_{s\to 0+}y(s,\varepsilon)=y(0,\varepsilon)\in(0,\infty),
\end{align}
it follows that
\[\lim_{s\to 0+} g(s,\varepsilon)=-\kappa y(0,\varepsilon)=:g(0,\varepsilon).\]
\par
We first consider the case $\kappa=\infty$. Since $g(s,\varepsilon)$ is increasing w.r.t $\varepsilon$, there exist some $\delta>0,~ \varepsilon_0\in(0,\eta)$ such that for all $s\in(0,\delta)$ and $\varepsilon\in[0,\varepsilon_0]$,
\[g(s,\varepsilon)<L_0-1.\]
Since $L_\varepsilon\geq L_0$, we get
\[L_\varepsilon=\sup_{s\in[\delta,1-\frac{a-\varepsilon}{\log m}]}g(s,\varepsilon),~\varepsilon \in[0,\varepsilon_0].\]
In the remaining of this paragraph, we only consider the case that $L<\infty$ and $\mathbb{P}(X=L)=0$. For other cases, the proof is easier. By similar arguments from (\ref{4rrw34rf}) to (\ref{6bu99}), there exists $\delta'>0$ such that for all $\varepsilon\in[0,\varepsilon_0]$,
$$
\sup_{s\in[\delta,1-\frac{a-\varepsilon}{\log m}]\atop \frac{y(s,\varepsilon)}{s}-\varepsilon\geq L-\delta'}g(s,\varepsilon)<L_0-1.
$$
Thus,
\[L_\varepsilon=\sup_{s\in B_{\varepsilon}}g(s,\varepsilon),~\varepsilon \in[0,\varepsilon_0],\]
where
\[B_{\varepsilon}:=\left\{s:\ \delta\leq s\leq 1-\frac{a-\varepsilon}{\log m}, \frac{y(s,\varepsilon)}{s}-\varepsilon\leq L-\delta'\right\}.\]
It is simple to see that
$$B_{\varepsilon}\downarrow B_0,~\text{as}~\varepsilon\to0+.$$
Since $g(s,0)$ is uniformly continuous on the compact set $B_{\varepsilon_0}$, for any $\varepsilon'>0$ there exists some $\eta>0$ such that for any $|s-s'| <\eta$ with $s,~s'\in B_{\varepsilon_0}$,
$$|g(s,0)-g(s',0)|<\frac{\varepsilon'}{2}.$$
By the definition of $B_\varepsilon$ and continuity of $y(s,\varepsilon)$, there exists some
$\eta_1\in(0,\varepsilon_0)$ such that for any $s\in B_{\varepsilon}$, $\varepsilon\in(0,\eta_1)$ there exists some $s'\in B_0$ satisfying
$$
|s-s'|<\eta.
$$
Since $g(s,\varepsilon)$ is continuous on the compact set $\{(s,\varepsilon): \varepsilon\in[0,\varepsilon_0], s\in B_{\varepsilon} \}$, there exists some
$\eta_2\in(0,\varepsilon_0)$ such that for $\varepsilon\in(0,\eta_2)$,
$$\sup_{s\in B_{\varepsilon}}[g(s,\varepsilon)-g(s,0)]<\frac{\varepsilon'}{2}.$$
Thus, it follows that for $\varepsilon\in(0,\eta_1\wedge\eta_2)$,
\[
\begin{split}
 L_\varepsilon-L_0&=\sup_{s\in B_{\varepsilon}}g(s,\varepsilon)-\sup_{s\in B_0}g(s,0)\cr
 &\leq\sup_{s\in B_{\varepsilon}}[g(s,\varepsilon)-g(s,0)]+\sup_{s\in B_{\varepsilon}}g(s,0)-\sup_{s\in B_0}g(s,0)\cr
 &\leq\frac{ \varepsilon'}{2}+\sup_{s\in B_{\varepsilon}}g(s,0)-\sup_{s\in B_0}g(s,0)\cr
 &\leq\varepsilon',
\end{split}
\]
which yields
\[\limsup_{\varepsilon\to 0}L_\varepsilon\leq L_0.\]
Combining with (\ref{dsfgdg}), the lemma follows if $\kappa=\infty$.
\par
In the case of $\kappa$ is finite, we first prove that $g(s,\varepsilon)$ is continuous on $\{0\}\times[0,\varepsilon_0]$. Then replace $\delta$ with $0$ in above steps, one can obtain the desired result. Without loss of generality, it suffices to show that $g(s,\varepsilon)$ is continuous at $(0,\varepsilon^*)$, where $\varepsilon^*\in(0,\varepsilon_0)$ (for $\varepsilon^*=0$ or $\varepsilon^*=\varepsilon_0$, we just need to replace the neighborhood of $\varepsilon^*$ with right or left neighborhood of $\varepsilon^*$). Write $h(s,\varepsilon):=\frac{y(s,\varepsilon)}{s}-\varepsilon$. Since $\lim_{t\to\infty}I(t)/t=\kappa$, there exists some $N>0$ such that for $t>N$,
$$\left|\frac{I(t)}{t}-\kappa\right|<\varepsilon'.$$
Since $y(s,\varepsilon)$ in continuous at $(0,\varepsilon^*)$ and $y(0,\varepsilon^*)>0$, there exist $\delta_1,\delta'_1,w>0$ such that $y(s,\varepsilon)>w$ for $0<s<\delta_1$ and $|\varepsilon-\varepsilon^*|<\delta'_1$. Thus, for $0<s<\frac{w}{N+1}$,
$$h(s,\varepsilon)=\frac{y(s,\varepsilon)}{s}-\varepsilon>\frac{w}{s}-1>N.$$
Since $y(s,\varepsilon)$ is uniform continuous on $\big[0,1-\frac{a}{\log m}\big]\times[0,\varepsilon_0]$, there exist some $\delta_2, \delta'_2>0$ such that for every $|s_1-s_2|<\delta_2$ and $|\varepsilon_1-\varepsilon_2|<\delta'_2$, $(s_1,\varepsilon_1)$, $(s_2,\varepsilon_2)\in\big[0,1-\frac{a}{\log m}\big]\times[0,\varepsilon_0]$,
$$|y(s_1,\varepsilon_1)-y(s_2,\varepsilon_2)|<\varepsilon'.$$
We take $\frac{I(\infty)}{\infty}=\kappa$ by convention. Therefore, for $0\leq s<\min\{\delta_1,\frac{w}{N+1},\delta_2,\varepsilon'\}$ and $|\varepsilon-\varepsilon^*|<\min\{\delta'_1,\delta'_2\}$
\begin{align}
&|g(s,\varepsilon)-g(0,\varepsilon^*)|\cr
&\leq|g(s,\varepsilon)-g(0,\varepsilon)|+|g(0,\varepsilon)-g(0,\varepsilon^*)|\cr
&=\Big|s\log m+\left[\kappa-I\left(h(s,\varepsilon)\right)/h(s,\varepsilon)\right]
(y(s,\varepsilon)-\varepsilon s)+\kappa\left[y(0,\varepsilon)-y(s,\varepsilon)+\varepsilon s\right]\Big|\cr
&~~~~~+\left|\kappa y(0,\varepsilon)-\kappa y(0,\varepsilon^*)\right|\cr
&\leq s\log m+\varepsilon'(x+1)+\kappa\left|y(0,\varepsilon)-y(s,\varepsilon)\right|+\kappa s+
\kappa\left| y(0,\varepsilon)-y(0,\varepsilon^*)\right|\cr
&\leq\varepsilon'(\log m+x+1+3\kappa).\nonumber
\end{align}
Therefore, $g(s,\varepsilon)$ is continuous at $(0,\varepsilon^*)$.
\end{proof}

\subsection{Proof of Theorem \ref{upth1}: lower bound}\label{subsec2}
By Lemma \ref{yohua}, to prove the lower bound, one only need to prove the following:
$$\liminf\limits_{n\rightarrow\infty}\frac{1}{n}\log\mathbb{P}(Z_n[xn,\infty)\geq e^{an})\geq-\inf_{s\in(0,1),y\in(0,x]:\atop~\frac{y}{s}>x^*,~\log m-I\left(\frac{x-y}{1-s}\right)\geq\frac{a}{1-s}}\left\{sI\left(\frac{y}{s}\right)-s\log m\right\}.$$
\par
The idea to the proof is as follows. Denote by $M_n:=\max\{S_u: u\in Z_n\}$ the maximum of the branching random walk at time $n$. Let $t_n=\lceil sn\rceil$, $s\in(0,1)$ and $u$ be the rightmost particle at time $t_n$. We let $M_{t_n}$ large enough so that the sub-BRW emanating from $u$ can normally produce $e^{an}$ descendants in $(nx,\infty)$ at time $n$. By this means, this problem transforms into the upper deviation probability of the  maximum, which has been well studied in \cite{Gantert18}. Finally, the desired lower bound follows by an optimization for $t_n$. Here, we note that
although, in \cite{Gantert18}, the step size is assumed to satisfy  $\mathbb{E}[e^{\lambda X}]<\infty$ for any $\lambda\in(-\ez,\ez)$ with some constant $\ez>0$,  in the lower bound of \cite[Theorem 3.2]{Gantert18}, one only need to assume $\mathbb{E} [e^{\lambda X}]<\infty$ for some $\lambda>0$. Hence,  in Theorem \ref{upth1}, we just assume the step size has positive exponential moment. As usual, we use $\lfloor x\rfloor$ ($\lceil x \rceil$) to stand for the largest integer not greater than $x$ (the smallest integer not less than $x$).
\begin{proof}
 Let $s\in(0,1)$, $y\in(0,x)$ satisfying
\begin{equation}\label{jmj}
\frac{y}{s}>x^*~\text{and}~\log m-I\left(\frac{x-y}{1-s}\right)\geq\frac{a}{1-s}.
\end{equation}
 Fix $\varepsilon\in(0,1)$. Let $t_n=\lceil sn\rceil$. By the Markov property,
\begin{align}\label{low1}
&\mathbb{P}(Z_n[xn,\infty)\geq e^{an})\cr
&\geq \mathbb{P}\left(M_{t_n}\geq (1+\varepsilon)yn\right)\mathbb{P}\left(Z_{n-t_n}(nx-(1+\varepsilon)yn,\infty)\geq e^{an}\right).
\end{align}
For the first factor on the r.h.s of (\ref{low1}), since $(1+\varepsilon)\frac{y}{s}>x^*$, by the proof of \cite[Theorem 3.2]{Gantert18}, we have for $n$ large enough
\begin{equation}\label{hjh}
\mathbb{P}\left(M_{t_n}\geq (1+\varepsilon)\frac{yn}{t_n}t_n\right)\geq\exp\left\{-(1+\varepsilon)\left(I\left(\frac{(1+\varepsilon)y}{s}\right)-\log m\right)t_n\right\}.
\end{equation}
\par
On the other hand, (\ref{jmj}) implies that
\begin{equation}\label{wews}
\frac{(x-(1+\varepsilon)y)}{1-s}<x^*~\text{and}~\log m-I\left(\frac{x-(1+\varepsilon)y}{1-s}\right)>\frac{a}{1-s}.
\end{equation}
Thus, for the second factor on the r.h.s of (\ref{low1}), since (\ref{wews}) holds, using Lemma \ref{a.s.} and the dominated convergence theorem, it follows that
\begin{equation}\label{wewe}
\lim\limits_{n\rightarrow\infty}\mathbb{P}\left(\frac{1}{n-t_n}\log Z_{n-t_n}\left(\frac{(x-(1+\varepsilon)y)n}{n-t_n}(n-t_n),\infty\right)\geq \frac{an}{n-t_n}\right)=1.
\end{equation}
Plugging (\ref{hjh}) and (\ref{wewe}) into (\ref{low1}) yields that for $n$ large enough
$$
\mathbb{P}(Z_n[xn,\infty)\geq e^{an})\geq 0.9 \exp\left\{-(1+\varepsilon)\left(I\left(\frac{(1+\varepsilon)y}{s}\right)-\log m\right)t_n\right\}.
$$
Taking limits yields
$$\liminf\limits_{n\rightarrow\infty}\frac{1}{n}\log\mathbb{P}(Z_n[xn,\infty)\geq e^{an})\geq-(1+\varepsilon)\left(sI\left(\frac{(1+\varepsilon)y}{s}\right)-s\log m\right).$$
Letting $\varepsilon\rightarrow0+$ and then optimizing for $s$ and $y$, we obtain
$$\liminf\limits_{n\rightarrow\infty}\frac{1}{n}\log\mathbb{P}(Z_n[xn,\infty)\geq e^{an})\geq-\inf_{s\in(0,1),y\in(0,x):\atop~\frac{y}{s}>x^*,~\log m-I\left(\frac{x-y}{1-s}\right)\geq\frac{a}{1-s}}\left\{sI\left(\frac{y}{s}\right)-s\log m\right\}
.$$
This completes the proof of the lower bound.
\end{proof}

\subsection{Proof of Theorem \ref{upth1}: upper bound}\label{subsec3}
In this subsection, we are going to prove the upper bound:
$$\limsup\limits_{n\rightarrow\infty}\frac{1}{n}\log\mathbb{P}(Z_n[xn,\infty)\geq e^{an})\leq -I(a,x)$$
\par
The idea to the proof is as follows. We divide the space-time into some paths. These paths can be classified into two categories: good path and bad path. Then, to estimate $\mathbb{P}(Z_n[xn,\infty)\geq e^{an})$ is reduced to consider the probability of the number of particles to follow a fixed path is large than $e^{an}$. For good path, the probability can be well-estimated by the Markov inequality, which exhibits an exponential decay. For bad path, the probability is essentially the upper deviation probability of an inhomogeneous Galton-Watson process, which has been well-handled in \cite[Proposition 2.1]{levelsetzhan}, and the probability turns out to be a double-exponential decay. Hence, the desired upper bound comes from an optimization for the good path.
 \begin{proof}Let $\delta\in(0,1)$. We discretize time by splitting time interval $[0,n]$ into intervals of length $\lfloor n^{\delta}\rfloor$. Let $s_0=0, s_i=n-(M-i)\lfloor n^{\delta}\rfloor$, $i=1,...,M$, where $M:=\lfloor n/\lfloor n^{\delta}\rfloor\rfloor+1$. Let $\theta$ be the constant in Theorem \ref{upth1} and  write $C=(2\log m+\log\mathbb{E}\left[e^{\theta X}\right]+I(a,x))/\theta$.

 \par
 Denote by $\text{supp} Z_n$ the support of $Z_n$. Set
 \begin{equation}\label{event1}
 E_1(n):=\left\{\text{supp} Z_i\subset(-\infty,Cn]~\text{for~every}~1\leq i\leq n\right\}.
 \end{equation}
Recall that $S_u$ is the position of particle $u$. Note that
 \begin{align}\label{343df}
 \mathbb{P}\left(E_1(n)^c\right)
 &\leq\mathbb{P}\left(\exists 1\leq i\leq n~\text{and}~u\in Z_i~\text{s.t.}~S_u\geq Cn\right)\cr
 &\leq\mathbb{E}\left[\sum^n_{i=1}\sum_{u\in Z_i}\ind_{\{S_u\geq Cn\}}\right]\cr
 &=\sum^n_{i=1}m^i\mathbb{P}(S_i\geq Cn),
 \end{align}
 where $S_i:=\sum_{\ell=1}^i X_\ell$ and $X_\ell$ is independent and identically distributed as the step size of the branching random walk.
 Since $\mathbb{E}\left[e^{\theta X}\right]\geq e^{\mathbb{E}\left[\theta X\right]}=1$, by Markov inequality, for any $1\leq i\leq n$,
 \begin{equation}\label{dfer4}
 \mathbb{P}(S_i\geq Cn)\leq \mathbb{E}\left[e^{\theta(S_i-Cn)}\right]=\mathbb{E}^i\left[e^{\theta X}\right]e^{-\theta Cn}\leq \mathbb{E}^n\left[e^{\theta X}\right]e^{-\theta Cn}.
 \end{equation}
 Plugging (\ref{dfer4}) into (\ref{343df}) yields
 $$\mathbb{P}\left(E_1(n)^c\right)\leq nm^n\mathbb{E}^n\left[e^{\theta X}\right]e^{-\theta Cn}=o\left(e^{-I(a,x)n}\right).$$
 \par
 Let $Z^u_{m}$ be the $m$-th generation of the sub-BRW emanating from particle $u$. Set
 $$E_2(n):=\left\{|Z^u_{s_{i+1}-s_{i}}|\leq n^2m^{\lfloor n^{\delta}\rfloor}~\text{for~all}~u\in Z_{s_i},~0\leq i\leq M-1\right\}.$$
 Write $v\succ u$ if $v$ is a descendant of $u$. Put
 $$D_n=\left\{v\succ u: u\in Z_{s_i},|Z^u_{s_{i+1}-s_{i}}|> n^2m^{\lfloor n^{\delta}\rfloor}~\text{for some}~0\leq i\leq M-1\right\}.$$
 Let $\bar{Z}_{s_i}$, $0\leq i\leq M$ be the branching random walk by cutting all the individuals in $D_n$. In other words, $\{\bar{Z}_{s_i}\}_{0\leq i\leq M}$ is an inhomogeneous branching random walk with point processes $Z_{s_1}(\cdot)\ind_{\left\{|Z_{s_1}|\leq n^2m^{\lfloor n^{\delta}\rfloor}\right\}}$ and $Z_{\lfloor n^{\delta}\rfloor}\ind_{\left\{|Z_{\lfloor n^{\delta}\rfloor}|\leq n^2m^{\lfloor n^{\delta}\rfloor}\right\}}$ as its evolutionary mechanism at time $0$ and after time $s_1$, respectively. It is simple to see that
 \begin{align}\label{5trgw2}
 E_2(n)=\{\bar{Z}_{s_i}=Z_{s_i},~0\leq i\leq M\}.
 \end{align}
According to \cite[Theorem 4]{athreya94}, if $\mathbb{E}[e^{\theta |Z_1|}]<\infty$ for some $\theta>0$, then there exists some $\theta_1>0$ such that $$\sup_{i\geq1}\mathbb{E}\left[e^{\theta_1W_i}\right]<\infty,$$
where $W_i:=|Z_i|m^{-i}$, $i\geq1$. Thus,
 \begin{align}\label{4tef3e}
\mathbb{P}\left(E_2(n)^c\right)&=\mathbb{P}\left(\exists 0\leq i\leq M-1~\text{and}~u\in Z_{s_i}~\text{s.t.}~|Z^u_{s_{i+1}-s_{i}}|> n^2m^{\lfloor n^{\delta}\rfloor}\right)\cr
&\leq\sum^{M-1}_{i=0}m^{s_i}\mathbb{P}\left(|Z_{\lfloor n^{\delta}\rfloor}|\geq n^2m^{\lfloor n^{\delta}\rfloor}\right)\cr
&\leq Mm^{s_M}\sup _{i\geq1}\mathbb{P}\left(W_{i}\geq n^2\right)\cr
&\leq nm^{n}\sup_{i\geq1}\mathbb{E}\left[e^{\theta_1W_i-\theta_1 n^2}\right]\cr
&\leq \sup_{i\geq1}\mathbb{E}\left[e^{\theta_1W_i}\right]nm^{n}e^{-\theta_1 n^2}=o\left(e^{-I(a,x)n}\right),
\end{align}
where the third inequality follows by Markov inequality.
\par
We write $v\prec u$ if $v$ is an ancestor of $u$. Define an event
\begin{equation}\label{event3}
E_3(n):=\left\{\forall~u\in Z_n, S_u\geq xn,~\text{and}~v\prec u,~v\in Z_{i},~1\leq i< n~\text{are~located~in}~[-Cn,\infty) \right\}.
\end{equation}
Then,
\[
\begin{split}
\mathbb{P}(E_3(n)^c)&=\mathbb{P}\left(\exists u\in Z_n~\text{and}~v\prec u~ \text{satisfying}~S_u\geq xn~\text{and}~S_v<-Cn\right)\cr
&\leq \mathbb{E}\left[\sum_{u\in Z_n}\sum_{i=1}^{n-1}\sum_{v\in Z_i
\atop v\prec u} \ind_{\{S_v<-Cn,\ S_u\geq xn\}}\right]\\
&\leq m^n\sum_{i=1}^{n-1}\mathbb{P}(S_{n-i}>(x+C)n)\\
&\leq m^n\sum_{i=1}^{n-1}\mathbb{E}^{n-i}\left[e^{\theta X}\right]e^{-\theta(x+C)n}\\
&\leq nm^{n}\mathbb{E}^{n}\left[e^{\theta X}\right]e^{-\theta(x+C)n}=o\left(e^{-I(a,x)n}\right).\\
\end{split}
\]
\par
Consequently,
$$\mathbb{P}\left(Z_n[xn,\infty)\geq e^{an}\right)\leq \mathbb{P}\left(Z_n[xn,\infty)\geq e^{an},E_1(n),E_2(n),E_3(n)\right)+o\left(e^{-I(a,x)n}\right).$$
Fix $\delta'\in(0,\delta)$ and $\varepsilon\in(0,(x/3)\wedge a\wedge\varepsilon^*)$, where $\varepsilon^*$ is the unique positive root of
$$
a=\log m+\varepsilon^*-I(x-3\varepsilon^*).
$$
So,
\begin{align}\label{54yfr4w}
a>\log m+\varepsilon-I(x-3\varepsilon).
\end{align}
We discretize space by splitting space interval $[-Cn,Cn]$ into intervals of length $ n^{\delta'}$. Let $x_k=kn^{\delta'}$, $-\lceil Cn^{1-\delta'}\rceil\leq k\leq \lceil Cn^{1-\delta'}\rceil$. We call
$$f:\{s_i:0\leq i\leq M\}\rightarrow\left\{x_k=kn^{\delta'}:-\lceil Cn^{1-\delta'}\rceil\leq k\leq \lceil Cn^{1-\delta'}\rceil\right\}$$
a path if
$$f(0)=0~\text{and}~f(s_M)\geq (x-\varepsilon)n.$$
For a branching random walk, a particle at time $s_i$ is said to follow a path $f$ until time $s_i$ if for all $0\leq j\leq i$ the ancestor of the particle at time $s_j$ lies in $\left[f(s_j)-n^{\delta'},f(s_j)+n^{\delta'}\right]$. Let
\begin{align}\label{4ree45trg}
Z_{s_i}(f):=\text{the number of particles following a path}~f~\text{until time}~s_i.
\end{align}
On the event $E_1(n)\cap E_3(n)$, we have
\begin{equation}\label{eq11}
Z_n[xn,\infty)\leq \sum_{f}Z_{s_M}(f)\leq\#(\text{paths})\max_{f}Z_{s_M}(f),
\end{equation}
where $\sum_{f}$ and $\max_f$, respectively, denote the sum and maximum over all possible paths $f$, and
$\#$(paths) stands for the total number of paths. Let $a'\in(0,a)$. Since $\#(\text{paths})\leq (2Cn^{1-\delta'}+2)^{n^{1-\delta}+1}=e^{o(n)}$ as $n\rightarrow\infty$, by (\ref{eq11}), it follows that for $n$ large enough, on the event $\{Z_n[xn,\infty)\geq e^{an}\}\cap E_1(n)\cap E_3(n)$, there exists a path $f$ such that $Z_{s_M}(f)\geq e^{a'n}$. Accordingly, for $n$ large enough,
 \begin{equation}\label{eq12}
 \mathbb{P}\left(Z_n[xn,\infty)\geq e^{an},E_1(n),E_2(n),E_3(n)\right)\leq e^{o(n)}\max_{f}\mathbb{P}\left(Z_{s_M}(f)\geq e^{a'n},E_2(n)\right).
 \end{equation}
 Since $I(a,x)$ is left continuous at $a$ (see Lemma \ref{yohua}), the proof of the upper bound is reduced to show the following:
 $$\limsup_{n\to \infty}\frac{1}{n}\max_f\log\mathbb{P}\left(Z_{s_M}(f)\geq e^{an},E_2(n)\right)\leq -I(a,x).$$
 To bound $\mathbb{P}\left(Z_{s_M}(f)\geq e^{an},E_2(n)\right)$, we distinguish two situations. A path $f$ is said to be good if
 there exists $1\leq i\leq M$ such that
 \begin{equation}\label{qwer}
 (n-s_i)\log m-I\left(\frac{f(s_M)-f(s_i)}{n-s_i}-\varepsilon\right)(n-s_i)\geq (a-\varepsilon)n.
 \end{equation}
 Otherwise, we say $f$ is a bad path.
 \par
  We claim that (\ref{qwer}) implies $f(s_i)>\varepsilon n$. In fact, if $f(s_i)\leq\varepsilon n$, then we have
 \begin{align}\label{yuiyt}
 I\left(\frac{f(s_M)-f(s_i)}{n-s_i}-\varepsilon\right)(n-s_i)&\geq I\left(\frac{f(s_M)-\varepsilon n}{n-s_i}-\varepsilon\right)(n-s_i)\cr
 &\geq I\left(\frac{(x-3\varepsilon)n}{n-s_i}\right)(n-s_i)\cr
 &\geq I\left(x-3\varepsilon\right)\frac{n}{n-s_i}(n-s_i)\cr
 &=I\left(x-3\varepsilon\right)n,
 \end{align}
 where the second inequality follows from the fact that $f(s_M)>(x-\varepsilon)n$ and the third inequality follows by the convexity of $I(\cdot)$. Since $a>\log m+\varepsilon-I(x-3\varepsilon)$, by (\ref{yuiyt}),
 \begin{align}
 I\left(\frac{f(s_M)-f(s_i)}{n-s_i}-\varepsilon\right)(n-s_i)+(a-\varepsilon)(n-s_i)&\geq I\left(x-3\varepsilon\right)n+(a-\varepsilon)(n-s_i)\cr
 &>(n-s_i)\log m,\nonumber
 \end{align}
 which contradicts with (\ref{qwer}). Thus, $f(s_i)>\varepsilon n$.
 \par
 We first consider the situation that $f$ is a good path. Observe that there exists some $N>0$ such that for $n>N$ and any good path $f$ with $i$ defined in (\ref{qwer}),
 \begin{align}
 \mathbb{P}\left(Z_{s_M}(f)\geq e^{an},E_2(n)\right)&\leq\mathbb{P}\left(Z_{s_i}(f)\geq 1\right)\cr
 &\leq m^{s_i}\mathbb{P}\left(|S_{s_i}-f(s_i)|\leq n^{\delta'}\right)\cr
 &\leq m^{s_i}\mathbb{P}\left(\frac{S_{s_i}}{s_i}>\frac{f(s_i)}{s_i}-\varepsilon\right)\cr
 &\leq \exp\left\{s_i\log m-I\left(\frac{f(s_i)}{s_i}-\varepsilon\right)s_i\right\}.\nonumber
 \end{align}
 Thus, for any good path $f$,
 \begin{align}\label{eq4}
 \mathbb{P}\left(Z_{s_M}(f)\geq e^{an},E_2(n)\right)&\leq\exp\bigg\{\sup_{s\in(0,n),y>\varepsilon n,z\geq (x-\varepsilon)n:\atop
 (n-s)\log m-I\left(\frac{z-y}{n-s}-\varepsilon\right)(n-s)\geq (a-\varepsilon)n} \left\{s\log m-I\left(\frac{y}{s}-\varepsilon\right)s\right\}\bigg\}\cr
 &=\exp\bigg\{n\sup_{s\in(0,1),y>\varepsilon,z\geq (x-\varepsilon):\atop
 (1-s)\log m-I\left(\frac{z-y}{1-s}-\varepsilon\right)(1-s)\geq (a-\varepsilon)} \left\{s\log m-I\left(\frac{y}{s}-\varepsilon\right)s\right\}\bigg\}.
\end{align}
Thus, by Lemma \ref{yohuaepsilon} and Lemma \ref{yohua}, uniformly in all good paths $f$, we have
$$\limsup_{n\to \infty}\frac{1}{n}\log\mathbb{P}\left(Z_{s_M}(f)\geq e^{an},E_2(n)\right)\leq -I(a,x).$$
\par
 Hence, to obatin the desired upper bound it suffices to show that for all bad paths $f$ uniformly,
\[\limsup_{n\to \infty}\frac{1}{n}\log\mathbb{P}(Z_{s_M}(f)\geq e^{an},\ E_2(n))=-\infty.\]
Let $\varepsilon'\in(0,\varepsilon/2)$. For any path $f$, define
\[\tau=\tau(f,n):=\inf\{i:\ 1\leq i\leq M,\ Z_{s_i}(f)\geq e^{\varepsilon'n}\},\ \inf\varnothing:=\infty.\]
On the event $\{Z_{s_M}(f)\geq e^{an}\}\cap E_2(n)$, we have $1\leq\tau\leq M-1$ and $Z_{s_\tau}(f)\leq n^2m^{\lfloor n^{\delta}\rfloor}e^{\varepsilon'n}$. Hence,
\begin{align}\label{1}
\mathbb{P}(Z_{s_M}(f)\geq e^{an},\ E_2(n))&\leq\mathbb{P}(Z_{s_M}(f)\geq e^{an},\ Z_{s_\tau}(f)\leq n^2 m^{\lfloor n^{\delta}\rfloor} e^{\varepsilon'n},\ E_2(n))\cr
&\leq \sum_{i=1}^{M-1}\mathbb{P}(Z_{s_M}(f)\geq e^{an},\ e^{\varepsilon'n}\leq Z_{s_i}(f)\leq n^2m^{\lfloor n^{\delta}\rfloor} e^{\varepsilon'n},\ E_2(n))\cr
&\leq\sum_{i=1}^{M-1}\sum_{\ell=\lfloor e^{\varepsilon'n}\rfloor}^{n^2m^{\lfloor n^{\delta}\rfloor} e^{\varepsilon'n}}\mathbb{P}(Z_{s_M}(f)\geq e^{an},~ Z_{s_i}(f)=\ell,\ E_2(n))\cr
&\leq\sum_{i=1}^{M-1}\sum_{\ell=\lfloor e^{\varepsilon'n}\rfloor}^{n^2m^{\lfloor n^{\delta}\rfloor} e^{\varepsilon'n}}\mathbb{P}(\bar{Z}_{s_M}(f)\geq e^{an},~ \bar{Z}_{s_i}(f)=\ell),
\end{align}
where $\bar{Z}_{s_i}(f)$ is defined similar to (\ref{4ree45trg}) (thus $\bar{Z}_{s_i}(f)=Z_{s_i}(f)$ on $E_2(n)$) and the last inequality follows from the following (see (\ref{5trgw2})):
 \begin{align}
 E_2(n)=\{\bar{Z}_{s_i}=Z_{s_i},~1\leq i\leq M\}.\nonumber
 \end{align}
\par
Next, we deal with the probability $\mathbb{P}(\bar{Z}_{s_M}(f)\geq e^{an},\ \bar{Z}_{s_i}(f)=\ell,\ E_2(n))$. Fix $1\leq i\leq M-1$. The sequence $(\bar{Z}_{s_{i+j}}(f), \ 0\leq j\leq M-i)$ can be written as
$$\bar{Z}_{s_{i+j+1}}(f)=\sum_{k=1}^{\bar{Z}_{s_{i+j}}(f)}\nu_k^{(j)}~\text{for}~0\leq j\leq M-i-1,$$
where $\nu_k^{(j)}$ is the number of particles in $[f(s_{i+j+1})-n^{\delta'},f(s_{i+j+1})+n^{\delta'}]$ at time $s_{i+j+1}$ generated by the $k$-th particle of $\bar{Z}_{s_{i+j}}(f)$. We define an inhomogeneous Galton-Watson process $(\tilde{Z}_{s_{i+j}}(f), 0\leq j\leq M-i)$ as follows. For $j=0$, $\tilde{Z}_{s_{i}}(f):=\bar{Z}_{s_{i}}(f)$. For $1\leq j\leq M-i$, we define it by induction. Let $u$ be the $k$th particle of $\tilde{Z}_{s_{i+j}}(f)$ with its position $S_u$. Let $\tilde{\nu}_k^{(j)}$ be the number of descendants of $u$ which lie in  $$H:=\Big[f(s_{i+j+1})-n^{\delta'}+S_u-(f(s_{i+j})+n^{\delta'}),f(s_{i+j+1})+n^{\delta'}+S_u-(f(s_{i+j})-n^{\delta'})\Big]$$
 at time $s_{i+j+1}$. For $0\leq j\leq M-i-1$, write
$$\tilde{Z}_{s_{i+j+1}}(f):=\sum_{k=1}^{\tilde{Z}_{s_{i+j}}(f)}\tilde{\nu}_k^{(j)}.$$
Note that if $S_u\in[f(s_{i+j})-n^{\delta'}, f(s_{i+j})+n^{\delta'}]$, then
$$H\supset \Big[f(s_{i+j+1})-n^{\delta'},f(s_{i+j+1})+n^{\delta'}\Big].$$
So, $\bar{Z}_{s_{i+j}}(f)\leq \tilde{Z}_{s_{i+j}}(f)$, $0\leq j\leq M-i$.
 Write $\Delta f(s_{i+j}):=f(s_{i+j+1})-f(s_{i+j})$. Note that $s_{k+1}-s_{k}=\lfloor n^\delta\rfloor$ for $k=1,...,M-1$. Thus, for $n$ large enough,
\[
\begin{split}
m_j&:=\mathbb{E}\left[\tilde{\nu}_k^{(j)}\right]\\
&=\mathbb{E}\left[|\bar{Z}^u_{\lfloor n^\delta\rfloor}|\right]\mathbb{P}\left(|S_{\lfloor n^\delta\rfloor}-\Delta f(s_{i+j})|\leq 2n^{\delta'}\right)\\
&\leq m^{\lfloor n^\delta\rfloor}\mathbb{P}\left(|S_{\lfloor n^\delta\rfloor}-\Delta f(s_{i+j})|\leq 2n^{\delta'}\right)\\
&\leq m^{\lfloor n^\delta\rfloor}\mathbb{P}\left(\frac{S_{\lfloor n^\delta\rfloor}}{\lfloor n^\delta\rfloor}>\frac{\Delta f(s_{i+j})- 2n^{\delta'}}{\lfloor n^\delta\rfloor}\right)\\
&\leq \exp\left\{\lfloor n^{\delta}\rfloor\log m-\lfloor n^\delta\rfloor I\left(\frac{\Delta f(s_{i+j})}{\lfloor n^{\delta}\rfloor}-\varepsilon\right)\right\},
\end{split}
\]
where the first inequality follows from the fact that
$$
\mathbb{E}\left[|\bar{Z}^u_{\lfloor n^\delta\rfloor}|\right]=\mathbb{E}\left[|Z_{\lfloor n^\delta\rfloor}|\ind_{\{|Z_{\lfloor n^\delta\rfloor}|\leq n^2m^{\lfloor n^{\delta}\rfloor}\}}\right]\leq m^{\lfloor n^\delta\rfloor}.
$$
Thus, for $0\leq k\leq M-i-1$, above yields
\begin{equation}\label{werw}
\prod_{j=k}^{M-i-1}m_j\leq\exp\left\{(M-i-k)\lfloor n^\delta\rfloor \log m- \lfloor n^\delta\rfloor \sum_{j=k}^{M-i-1}I\left(\frac{\Delta f(s_{i+j})}{\lfloor n^{\delta}\rfloor}-\varepsilon\right)\right\}.
\end{equation}
Since $I(\cdot)$ is convex, we have
\begin{align}
 &\frac{1}{M-i-k}\sum_{j=k}^{M-i-1} I\left(\frac{\Delta f(s_{i+j})}{\lfloor n^{\delta}\rfloor}-\varepsilon\right)\cr
 &\geq I\left(\sum_{j=k}^{M-i-1}\frac{\Delta f(s_{i+j})-\varepsilon\lfloor n^{\delta}\rfloor}{(M-i-k)\lfloor n^{\delta}\rfloor}\right)=
I\left(\frac{f(s_{M})-f(s_{i+k})}{(M-i-k)\lfloor n^\delta\rfloor}-\varepsilon\right).\nonumber
\end{align}
Using above inequality, (\ref{werw}) yields
\begin{align}\label{ytitjfs}
&\max_{0\leq k\leq M-i-1}\prod_{j=k}^{M-i-1}m_j\cr
&\leq \exp\left\{(M-i-k)\lfloor n^\delta\rfloor \log m- (M-i-k)\lfloor n^\delta\rfloor I\left(\frac{f(s_{M})-f(s_{i+k})}{(M-i-k)\lfloor n^\delta\rfloor}-\varepsilon\right)\right\}\cr
&=\exp\left\{(n-s_{i+k})\log m-(n-s_{i+k})I\left(\frac{f(s_M)-f(s_{i+k})}{n-s_{i+k}}-\varepsilon\right)\right\}\cr
&\leq e^{(a-\varepsilon)n},
\end{align}
where the last inequality comes from the definition of bad path.
 \par
Let $\alpha>1$ and $h>0$ be some constants satisfying $\alpha+h<e^{\varepsilon-2\varepsilon'}$.  There exists a positive constant $r$ depending only on $\alpha$ such that for every $y\in[0,r]$, we have $e^{y}-1\leq \alpha y$. Set
\begin{align}\label{4tgft4rt434}
\lambda_j:=m^{-2n^\delta},
\end{align}
which implies $\lambda_j\tilde{\nu}^{(j)}\leq n^2 m^{- n^\delta}$. Thus, for $n$ large enough,
\[
\begin{split}
\mathbb{E}\left[e^{\lambda_j \tilde{\nu}^{(j)}}\right]\leq \mathbb{E}[1+\alpha \lambda_j\tilde{\nu}^{(j)}]\leq1+\alpha \lambda_jm_j\leq e^{\alpha \lambda_j m_j}.
\end{split}
\]
Thus, by \cite[Proposition 2.1]{levelsetzhan}, there exists some $N'>0$ such that for $n>N'$, $i=1,2,...,M-1$ and $\lfloor e^{\varepsilon'n}\rfloor\leq \ell\leq n^2m^{\lfloor n^{\delta}\rfloor} e^{\varepsilon'n}$,
\begin{align}\label{fgegtw}
& \mathbb{P}(\bar{Z}_{s_M}(f)\geq e^{an},\ \bar{Z}_{s_i}(f)=\ell)\cr
&\leq \mathbb{P}(\tilde{Z}_{s_{M}}(f)\geq e^{an},\ \tilde{Z}_{s_{i}}(f)=\ell)\cr
&\leq\mathbb{P}\left(\left.\tilde{Z}_{s_M}(f)\geq \max\bigg\{\ell,\ (\alpha+h)^n\ell \max_{0\leq k< M-i}\prod_{j=k}^{M-i-1}m_j\bigg\}\right|\tilde{Z}_{s_i}(f)=\ell\right)\cr
&\leq n\exp\left(-\frac{ \ell h}{\alpha+h}m^{-2n^\delta}+m^{-2n^\delta}\right),
\end{align}
where the second inequality follows by the fact that from (\ref{ytitjfs}), we have
\[(\alpha+h)^n\ell \max_{0\leq k< M-i}\prod_{j=k}^{M-i-1}m_j\leq e^{(\varepsilon-2\varepsilon')n} n^2e^{\lfloor n^\delta\rfloor\log m+\varepsilon' n}e^{n(a-\varepsilon)}\leq e^{an}.\]
\par
Plugging (\ref{fgegtw}) into (\ref{1}) yields that for $n>N'$ and every bad path $f$,
\begin{equation*}
\mathbb{P}(Z_{s_M}(f)\geq e^{an},\ E_2(n))\leq n^4 e^{\lfloor n^\delta\rfloor\log m+\varepsilon' n}\exp\left(-\frac{h\lfloor e^{\varepsilon'n}\rfloor}{\alpha+h}m^{-2n^\delta}+m^{-2n^{\delta}}\right).
\end{equation*}
Thus,
\[\limsup_{n\to\infty}\frac{1}{n}\log\mathbb{P}(Z_{s_M}(f)\geq e^{an},\ E_2(n))=-\infty\]
uniformly in all bad paths $f$. Therefore, the desired upper bound follows.
\par
\end{proof}
\textbf{Proof of Remark \ref{remark1}} The proof is almost the same as the proof of Theorem \ref{upth1}. The mainly change is to use \cite[Theorem 3]{Denisov GW heavy tail} in (\ref{4tef3e}).
\section{Proof of Theorem \ref{12defedi}} \label{sec2}
In this section, we are going to prove
\begin{align}
0&<\liminf_{n\to\infty}\frac{1}{n}\log\left[-\log\mathbb{P}(Z_n([xn,\infty))\geq e^{an})\right]\cr
&\leq\limsup_{n\to\infty}\frac{1}{n}\log\left[-\log\mathbb{P}(Z_n([xn,\infty))\geq e^{an})\right]\leq c^*\log b.\nonumber
\end{align}

For the lower bound, the strategy is to force every particle to produce exactly $b$ offsprings and have displacement near $L$ before time $\lfloor c^*n\rfloor$. $c^*$ is chosen such that particles at time $\lfloor c^*n\rfloor$
can naturally produce $e^{an}/b^{\lfloor c^*n\rfloor}$ particles in $[xn,\infty)$ at time $n$.  The proof of upper bound is basically the same as Theorem \ref{upth1}. However, since under the assumptions of Theorem \ref{12defedi}, the probability that the branching random walk follows a good path is $0$ (see (\ref{eq7}) below), the decay rate comes from the bad path.
\begin{proof}
\textbf{Lower bound.} Recall that $b=\min\{k\geq m: p_k>0\}$. Let
$$u(c):=\log m-I\left(L+\frac{x-L}{1-c}\right)-\frac{a-\log b}{1-c},~c\in\left[0,\frac{x}{L}\right].$$
It is easy to see that $u(c)$ is strictly increasing and continuous on $[0,\frac{x}{L}]$. Moreover,
\begin{align}
u(0)&=\log m-I(x)-a+\log b< \log b,\cr
 u\left(\frac{x}{L}\right)&=\log m-\frac{a-\log b}{1-\frac{x}{L}}> \log m-(a-\log b)>\log b.\nonumber
\end{align}
Thus, the equation $u(c)=\log b$ has an unique solution on $(0,\frac{x}{L})$, denoted as $c^*$. Fix $\epsilon\in(0,\frac{x}{L}-c^*).$
Let $t_n:=\lfloor(c^*+\epsilon)n\rfloor.$ Since $u(c^*+\epsilon)>0$, there exists some $\eta\in(0,L)$ such that
$$\log m-I\left(\frac{x-(L-\eta)(c^*+\epsilon)}{1-(c^*+\epsilon)}\right)>\frac{a-(c^*+\epsilon)\log b}{1-(c^*+\epsilon)}.$$
Thus, by Lemma \ref{a.s.}, we have
\begin{align}\label{4res3ews}
\lim_{n\to\infty}\mathbb{P}\left(Z_{n-t_n}[nx-(L-\eta)t_n,\infty)>\frac{e^{an}}{b^{t_n}}\right)=1.
\end{align}
Recall that for a particle $u$, $(Z^u_n,~n\geq0)$ stands for the sub-BRW emanating from $u$. For a particle $v$, let $X_v$ and $|v|$ be the displacement and the generation of it, respectively. Since the branching and spatial motion are independent, we have for $n$ large enough,
\begin{align}
&\mathbb{P}(Z_n[xn,\infty)\geq e^{an})\cr
&\geq\mathbb{P}\left(\sum_{u\in Z_{t_n}}Z^u_{n-t_n}[nx-(L-\eta)t_n,\infty)\geq e^{an}, X_v\geq L-\eta, 1\leq |v|\leq t_n, |Z_{t_n}|=b^{t_n}\right)\cr
&\geq\mathbb{P}\left(Z^u_{n-t_n}[nx-(L-\eta)t_n,\infty)>\frac{e^{an}}{b^{t_n}},~\forall u\in Z_{t_n}, |Z_{t_n}|=b^{t_n}\right )\mathbb{P}\left(X\geq L-\eta\right)^{b+b^2+...+b^{t_n}}\cr
&=\mathbb{P}\left(Z_{n-t_n}[nx-(L-\eta)t_n,\infty)>\frac{e^{an}}{b^{t_n}}\right)^{b^{t_n}} \mathbb{P}\left(|Z_{t_n}|=b^{t_n}\right)   \mathbb{P}\left(X\geq L-\eta\right)^{b+b^2+...+b^{t_n}}\cr
&\geq(0.9)^{b^{t_n}}(p_b)^{1+b+b^2+...+b^{t_n-1}}\mathbb{P}\left(X\geq L-\eta\right)^{b+b^2+...+b^{t_n}},\nonumber
\end{align}
where the last inequality follows from (\ref{4res3ews}). Taking limits yields that

$$
\limsup_{n\rightarrow\infty}\frac{1}{n}\log[-\log\mathbb{P}(Z_n[xn,\infty)\geq e^{an})]\leq (c^*+\epsilon)\log b.
$$
The desired lower bound follows by letting $\epsilon\to0+$.
\par
\textbf{Upper bound.} Recall that $y_s$ is defined in Remark \ref{4r4rdef}. In the following proofs, to emphasize the dependence on $a$, we write $y_s(a):=y_s$.  We first show that both Theorem \ref{12defedi} (i) and (ii) imply
\begin{equation}\label{eq2}
\inf_{s\in(0,1-\frac{a}{\log m}]}\frac{y_s(a)}{s}>L.
\end{equation}
Obviously, (ii) yields above.
\par
In the next, we will show that (i) (i.e. $x^*=L$) also implies (\ref{eq2}). By Lemma \ref{sratefunction} (ii), we have $\lim_{x\to x^*-}I(x)=I(x^*)$ and $I(x^*)\leq \log m$.
Similar to the arguments from (\ref{4rewe22d})-(\ref{efrefe}), we get
\begin{align}\label{4rfdf3d}
\frac{y_s(a)}{s}>x^*=L\ \ \text{for }s\in\Big(0,1-\frac{a}{\log m}\Big].
\end{align}
Note that $\lim_{s\to 0+}y_s(a)/s=\infty$ and $y_s(a)/s$ is continuous w.r.t. $s$ on $(0,1-\frac{a}{\log m}]$. This, combined with (\ref{4rfdf3d}), implies (\ref{eq2}).

\par
Fix $\bar{a}\in((\log m-I(x))^+,a)$. By $\lim_{s\to 0+}y_s(\bar{a})/s=\infty$, there exists $\eta\in(0,1-\frac{a}{\log m})$ such that
\[\inf_{s\in(0,\eta)}\frac{y_s(\bar{a})}{s}>L.\]
Moreover, since $y_s(c)$ is increasing w.r.t. $c$, for any $c\in(\bar{a},a)$,
\begin{equation}\label{eq9}
\inf_{s\in(0,\eta)}\frac{y_s(c)}{s}\geq\inf_{s\in(0,\eta)} \frac{y_s(a)}{s}>x^*=L.
\end{equation}
Note that $y_s(c)/s$ is uniformly continuous on the compact set $\{(s,c):s\in[\eta,1-\frac{c}{\log m}],\ c\in[\bar{a},a]\}$. This, together with (\ref{eq2}), implies that there exists $a^*\in[\bar{a},a)$ such that for any $c\in(a^*,a)$,
\begin{equation}\label{eq10}
\inf_{s\in[\eta,1-\frac{c}{\log m}]}\frac{y_s(c)}{s}>L.
\end{equation}
By (\ref{eq9}) and (\ref{eq10}), it follows that for any $c\in(a^*,a)$,
\[l:=\inf_{s\in(0,1-\frac{c}{\log m}]}\frac{y_s(c)}{s}>L.\]
\par
Fix $c\in(a^*,a)$. By (\ref{eq3}), there exist $\rho>0$ and $\ez_0>0$ such that for any $s\in(0,\rho)$ and $\ez\in[0,\ez_0]$,
\begin{align}\label{34rfddf4}
\frac{y(s,\ez)}{s}>\frac{L+l}{2},
\end{align}
where $y(s,\ez)$ is defined in (\ref{23r4redr})-(\ref{eq3}) by replacing $a$ with $c$.
Because $y(s,\ez)/s$ is uniformly continuous on $B:=\{(s,\ez):\rho\leq s\leq 1-\frac{c-\ez}{\log m},\ 0\leq\ez\leq \ez_0\}$, there exists $\ez_2\in(0,\min\{(l-L)/2,\ez_0\})$ such that for any $\ez\in(0,\ez_2)$,
\[\inf_{s\in[\rho,1-\frac{c-\ez}{\log m}]} \frac{y(s,\ez)}{s}-\ez>l-\ez_2-\ez>L.\]
Thus, by (\ref{34rfddf4}), for any $\ez\in(0,\ez_2),$
\begin{align}\label{4r4eded3a}
\inf_{s\in(0,1-\frac{c-\ez}{\log m}]} \frac{y(s,\ez)}{s}-\ez>L.
\end{align}
\par
In the next, the proof is similar to those of Theorem \ref{upth1}. So, we only present some necessary modifications. In the definitions of $E_1(n)$ and $E_3(n)$ (see (\ref{event1}) and (\ref{event3})), we replace $C$ with $L$. Obviously,
\[\mP(E_1(n)^c)=\mP(E_3(n)^c)=0.\]
 Define
\[E_2(n):=\left\{|Z^u_{s_{i+1}-s_{i}}|\leq e^{\ez_1n}m^{\lfloor n^\delta \rfloor} \text{ for all } u\in Z_{s_i},\ 0\leq i\leq M-1\right\},\]
where $\ez_1$ is a positive number which will be determined later on.
Then, similar to (\ref{4tef3e}),
\[\mP(E_2(n)^c)\leq \sup_{i\geq 1} \mE\left[ e^{\theta_1 W_i}\right] nm^n e^{-\theta_1 e^{\ez_1 n}}.\]
Thus,
\begin{align}\label{eq5}
&\mP(Z_n[xn,\infty)\geq e^{an})\cr
&=\mP(Z_n[xn,\infty)\geq e^{an},E_1(n),E_2(n),E_3(n))+\mP(Z_n[xn,\infty)\geq e^{an},E_1(n),E_2(n)^c,E_3(n))\cr
&\leq\sup_{i\geq1} \mE\left[ e^{\theta_1 W_i}\right] nm^n e^{-\theta_1 e^{\ez_1 n}}+\mP(Z_n[xn,\infty)\geq e^{an},E_1(n),E_2(n),E_3(n)).
\end{align}
\par
It suffices to estimate $\mP(Z_n[xn,\infty)\geq e^{an},E_1(n),E_2(n),E_3(n))$. Similar to the arguments from (\ref{eq11})-(\ref{eq12}), there exists  $c\in(a^*,a)$ such that for $n$ large enough,
\begin{equation}\label{eq6}
\mP(Z_n[xn,\infty)\geq e^{an},E_1(n),E_2(n),E_3(n))\leq e^{o(n)}\max_{f}\mP(Z_{s_M}(f)\geq e^{cn},E_2(n)).
\end{equation}

\par
By (\ref{eq13}) and (\ref{eq4}), for any $\ez\in(0,\ez_2)$, there exists $N_1>0$ such that for any $n>N_1$ and any good path $f$,
 \begin{align}\label{eq7}
 \mP(Z_{s_M}(f)\geq e^{cn},E_2(n))&\leq\exp\bigg\{n\sup_{s\in(0,1),y>\varepsilon,z\geq (x-\varepsilon):\atop
 (1-s)\log m-I\left(\frac{z-y}{1-s}-\varepsilon\right)(1-s)\geq (c-\varepsilon)} \left\{s\log m-I\left(\frac{y}{s}-\varepsilon\right)s\right\}\bigg\}\cr
 &=\exp\bigg\{n\sup_{s\in\left(0,1-\frac{c-\varepsilon}{\log m}\right]} \left\{s\log m-I\left(\frac{y(s,\varepsilon)}{s}-\varepsilon\right)s\right\}\bigg\}\cr
 &=0,
 \end{align}
where the last inequality follows from (\ref{4r4eded3a}) and Lemma \ref{sratefunction} (ii).
 \par

  In the following, we consider the case that $f$ is a bad path. We replace the definition of $\lambda_j$ (see (\ref{4tgft4rt434})) with $\lambda_j:=e^{-2\ez_1 n}.$ Fix $\ez\in(0,\ez_2)$, for any $\ez'\in(0,\frac{\ez\wedge c}{2})$ and $\ez_1\in(0,\ez'/2)$, similar to the arguments from (\ref{1}) to (\ref{fgegtw}), for $n$ large enough and any bad path $f$,
\begin{align}
\mP&(Z_{s_M}(f)\geq e^{cn}, E_2(n))\cr
\leq&\sum_{i=1}^M \sum_{\ell=\lfloor e^{\ez' n}\rfloor}^{e^{\ez_1 n}m^{\lfloor n^{\delta}\rfloor} e^{\ez' n}}
 \mathbb{P}(\tilde{Z}_{s_{M}}(f)\geq e^{cn},\ \tilde{Z}_{s_{i}}(f)=\ell)\cr
\leq&\sum_{i=1}^M \sum_{\ell=\lfloor e^{\ez' n}\rfloor}^{e^{\ez_1 n}m^{\lfloor n^{\delta}\rfloor} e^{\ez' n}}
\mathbb{P}\left(\left.\tilde{Z}_{s_M}(f)\geq \max\bigg\{\ell,\ (\alpha+h)^n\ell \max_{0\leq k< M-i}\prod_{j=k}^{M-i-1}m_j\bigg\}\right|\tilde{Z}_{s_i}(f)=\ell\right)\cr
\leq&\sum_{i=1}^M \sum_{\ell=\lfloor e^{\ez' n}\rfloor}^{e^{\ez_1 n}m^{\lfloor n^{\delta}\rfloor} e^{\ez' n}}
n\exp\left(-\frac{\ell h}{\alpha+h}e^{-2\ez_1 n}+e^{-2\ez_1 n}\right),\nonumber
\end{align}
where the second inequality follows from the facts that $\ell\leq e^{\ez_1 n}m^{\lfloor n^{\delta}\rfloor} e^{\ez' n}<e^{cn}$ and
\[
\begin{split}
(\alpha+h)^n\ell\max_{0\leq k <M-i}\prod_{j=k}^{M-i-1}m_j&\leq e^{(\ez-2\ez')n} e^{\ez_1 n}e^{\lfloor n^\delta\rfloor\log m+\ez' n} e^{n(c-\ez)}\\
&=e^{(\ez_1-\ez')n}e^{\lfloor n^\delta\rfloor\log m}e^{cn}<e^{cn}.
\end{split}
\]
So, for $n$ large enough and any bad path $f$,
\[\mP(Z_{s_M}(f)\geq e^{cn}, E_2(n))\leq n^2 e^{\ez_1 n} e^{\lfloor n^\delta\rfloor\log m+\ez' n}\exp\left(-\frac{h\lfloor e^{\ez' n}\rfloor}{\alpha+h}e^{-2\ez_1 n}+e^{-2\ez_1 n}\right).\]
This, combined with (\ref{eq5}), (\ref{eq6}) and (\ref{eq7}), yields that
\[\liminf_{n\to\infty}\frac{1}{n}\log[-\log\mP(Z_n[xn,\infty)\geq e^{an})]\geq \min\{\ez_1,\ez'-2\ez_1\}>0.\]
\end{proof}

\section{Proof of Theorem \ref{remark2}} \label{sec3}
In this section, we deal with the case of $a\geq \log m$. We are going to prove
\[
\begin{split}
a-\log m&\leq\liminf_{n\to\infty}\frac{1}{n}\log[-\log\mP(Z_n([xn,\infty))\geq e^{an})]\\
&\leq \limsup_{n\to\infty}\frac{1}{n}\log[-\log\mP(Z_n([xn,\infty))\geq e^{an})]\leq\log \alpha.
\end{split}
\]
\par
For the lower bound, the strategy is to force all particles in the first $n-1$ generations to produce exactly $\alpha$ offsprings and have
displacements larger than $x$. The upper bound follows by the upper deviation probability of the total population $|Z_n|$.
\begin{proof}
\textbf{Lower bound.} Recall that $X_v$ stands for the displacement of $v$. Since $a\in[\log m,\log\mu]\cap\mathbb{R}$, $\alpha=\inf\{k:p_{k}>0,k\geq e^a\}\in[m,\infty)$, it is simple to see that
\[
\begin{split}
\mP(Z_n[xn,\infty)\geq e^{an})&\geq\mP(|Z_n|={\alpha}^n,\forall u\in Z_n, S_u\geq xn)\\
&\geq\mP\left(|Z_n|={\alpha}^n,\forall v\in Z_i, 1\leq i\leq n, X_v\geq x\right)\\
&\geq p_{\alpha}^{\sum_{i=0}^{n-1}{\alpha}^i}\mP(X \geq x)^{\sum_{i=1}^n {\alpha}^i}\\
&=p_{\alpha}^{\frac{{\alpha}^n-1}{{\alpha}-1}}\mP(X\geq x)^{\frac{{\alpha}^{n+1}-1}{\alpha-1}},
\end{split}
 \]
which implies
\[\limsup_{n\to\infty}\frac{1}{n}\log [-\log \mP(Z_n[xn,\infty)\geq e^{an})]\leq \log \alpha.\]
\textbf{Upper bound.} From \cite[Theorem 4]{athreya94}, there exists some $\theta_1>0$ such that $$\sup_{i\geq1}\mathbb{E}\left[e^{\theta_1W_i}\right]<\infty,$$
where $W_i=|Z_i|m^{-i}$, $i\geq1$. Thus, by Markov inequality,

\begin{align}
\mP(Z_n[xn,\infty)\geq e^{an})&\leq \mP(|Z_n|\geq e^{an})\cr
&=\mP\left(W_n\geq e^{(a-\log m)n}\right)\cr
&\leq\mE\left[e^{\theta_1 W_n}\right]e^{-\theta_1e^{(a-\log m)n}}.\nonumber
\end{align}
Taking limits yields that
$$
\liminf_{n\to\infty}\frac{1}{n}\log [-\log \mP(Z_n[xn,\infty)\geq e^{an})]\geq a-\log m.
$$
\end{proof}
\textbf{Proof of Remark \ref{remark3}.} Fix $a'\in((\log m-I(x))^+,a)$. Since $\lim_{s\to 0+}y_s(a')/s=\infty$(the definition of $y_s(a)$ is in Remark \ref{4r4rdef}) and $L<\infty$, there exists $\delta\in(0,1-\frac{a'}{\log m})$ such that for any $s\in(0,\delta)$, $y_s(a')/s>2L.$ Because $y_s(l)$ is increasing w.r.t. $l$, there exists $a^*\in(a',a)$ satisfying $1-\frac{a^*}{\log m}<\delta$. Thus,
\[\inf_{s\in(0,1-\frac{a^*}{\log m}]}\frac{y_s(a^*)}{s}>L.\]
Since $(\log m-I(x))^+<a^*<a=\log m$, the remark follows by using Theorem \ref{12defedi}.
\section{Proof of Theorem \ref{upth2}}\label{sec4}
In this section, we assume that $\mathbb{P}(|Z_1|>y)=\Theta(1)y^{-\beta}$ for some $\beta>1$ and $X$ is the standard normal random variable. Suppose that~$x>0$ and $a\in((\log m-\frac{x^2}{2})^+,\infty)$. We are going to prove Theorem \ref{upth2}. The idea to the proof is as follows. In the lower bound, we force particles located in high positions at time $n-1$ to produce more than $e^{an}$ children. Meanwhile, we force their children to have no big negative displacement. Naturally, the upper bound comes by arguing that if particles located in high positions at time $n-1$ produce less than $e^{an}$ children, then the event $\{Z_n[xn,\infty)\geq e^{an}\}$ happens with a negligible probability.
\begin{proof}
\textbf{Lower bound.}
We first consider the case of~$x\in(0,x^*)$. Fix~$\varepsilon\in(0,(x^*-x)/2)$. Recall that for particles $v$ and $u$, $X_v$ stands for the displacement of $v$ and $v\succ u$ means $v$ is a descendant of $u$. Define the event
$$ E_n:=\left\{\forall u\in Z_{n-1}, S_u\geq n(x+\varepsilon),~\text{we~have}~X_v\geq-\varepsilon n~\text{for}~v\succ u,~v\in Z_n\right\}.$$
Let~$k_u$ be the children number of~$u$. Let~$H_n:=\sum^n_{i=1} \left(|Z^i_1|-m\right),~n\geq1$, where~$|Z^i_1|,~i\geq 1$ are i.i.d copies of~$|Z_1|$. Recall that $S_u$ is the position of particle $u$. Write $Z_n(x,\infty)=Z_n(x)$ for short. It is easy to obatin
\begin{align}\label{ulhoff}
&\mathbb{P}(Z_n[xn,\infty)\geq e^{an})\cr
&\geq\mathbb{P}\Big(\sum_{u\in Z_{n-1}\atop S_u\geq n(x+\varepsilon)}k_u\geq e^{an}, E_n \Big)\cr
&\geq\mathbb{P}\Big(e^{3an}>\sum_{u\in Z_{n-1}\atop S_u\geq n(x+\varepsilon)}k_u\geq e^{an},\lfloor e^{2an}\rfloor\geq Z_{n-1}(n(x+\varepsilon))\geq \lfloor e^{\left(\log m-I(x+2\varepsilon)\right)n}\rfloor, E_n \Big)\cr
&\geq \Bigg[\mathbb{P}\bigg(\sum_{u\in Z_{n-1}\atop S_u\geq n(x+\varepsilon)}k_u\geq e^{an},\lfloor e^{2an}\rfloor\geq Z_{n-1}(n(x+\varepsilon))\geq \lfloor e^{\left(\log m-I(x+2\varepsilon)\right)n}\rfloor\bigg)-\cr
&~~~~~\mathbb{P}\bigg(\sum_{u\in Z_{n-1}\atop S_u\geq n(x+\varepsilon)}k_u\geq e^{3an},\lfloor e^{2an}\rfloor\geq Z_{n-1}(n(x+\varepsilon))\geq \lfloor e^{\left(\log m-I(x+2\varepsilon)\right)n}\rfloor\bigg)\Bigg]\mathbb{P}(X>-\varepsilon n)^{e^{3an}}\cr
&\geq\Bigg[\mathbb{P}\left(H_{\lfloor e^{\left(\log m-I(x+2\varepsilon)\right)n}\rfloor}\geq e^{an}-m\lfloor e^{\left(\log m-I(x+2\varepsilon)\right)n}\rfloor\right)-\mathbb{P}\left(H_{\lfloor e^{2an}\rfloor}\geq e^{3an}-m\lfloor e^{2an}\rfloor\right)\Bigg]\times\cr
&~~~~\mathbb{P}\left(\lfloor e^{2an}\rfloor\geq Z_{n-1}(n(x+\varepsilon))\geq \lfloor e^{\left(\log m-I(x+2\varepsilon)\right)n}\rfloor\right)\mathbb{P}(X>-\varepsilon n)^{e^{3an}}.
\end{align}
For the first factor on the right hand side of~(\ref{ulhoff}) , by~\cite[Theorem~1]{Nagaev82}
\begin{align}\label{uytsd1}
&\mathbb{P}\left(H_{\lfloor e^{\left(\log m-I(x+2\varepsilon)\right)n}\rfloor}\geq e^{an}-m\lfloor e^{\left(\log m-I(x+2\varepsilon)\right)n}\rfloor\right)-\mathbb{P}\left(H_{\lfloor e^{2an}\rfloor}\geq e^{3an}-m\lfloor e^{2an}\rfloor\right)\cr
&\geq 0.9e^{-\beta an}e^{\left(\log m-I(x+2\varepsilon)\right)n}-2e^{-3\beta an}e^{2an}\cr
&\geq 0.8e^{-\left[(\beta-1)a+a-(\log m-I(x+2\varepsilon))\right]n}.
\end{align}
For the second factor on the right hand side of~(\ref{ulhoff}), since~$\lim_{n\rightarrow\infty}\frac{1}{n}\log Z_n[xn,\infty)=\log m-I(x)$ a.s., it follows that for~$n$ large enough,
\begin{align}\label{uytsd2}
&\mathbb{P}\left(\lfloor e^{2an}\rfloor\geq Z_{n-1}(n(x+\varepsilon))\geq \lfloor e^{\left(\log m-I(x+2\varepsilon)\right)n}\rfloor\right)\cr
&=\mathbb{P}\left(\frac{\log\lfloor e^{2an}\rfloor}{n}\geq\frac{\log Z_{n-1}(n(x+\varepsilon))}{n}\geq \frac{\log\lfloor e^{\left(\log m-I(x+2\varepsilon)\right)n}\rfloor}{n}\right)\geq0.9.
\end{align}
From~\cite[p349]{Morters2010}, for~$y>0$,
$$\frac{y}{y^2+1}\frac{1}{\sqrt2\pi}{e^{-\frac{y^2}{2}}}\leq\mathbb{P}(X>y)\leq\frac{1}{y}\frac{1}{\sqrt2\pi}{e^{-\frac{y^2}{2}}}.$$
Thus, for the third factor on the right hand side of~(\ref{ulhoff}) , we have for~$n$ large enough,
\begin{align}\label{uytsd3}
\mathbb{P}(X>-\varepsilon n)^{e^{3an}}&=(1-\mathbb{P}(X>\varepsilon n))^{e^{3an}}\cr
&\geq (1-e^{-\varepsilon^2n^2/4})^{e^{3an}}\geq 0.9.
\end{align}
Plugging~(\ref{uytsd1}),~(\ref{uytsd2}) and~(\ref{uytsd3}) into~(\ref{ulhoff}) yields that
$$
\liminf_{n\rightarrow\infty}\frac{1}{n}\log\mathbb{P}(Z_n[xn,\infty)\geq e^{an})\geq-\left[(\beta-1)a+a-(\log m-I(x+2\varepsilon))\right].
$$
Since~$X$ is the standard normal random variable, ~$I(x)=x^2/2$. The desired lower bound follows by letting~$\varepsilon\rightarrow0+$.
\par
 In the next, we consider the case of ~$x\geq x^*$. Let $w$ be the rightmost particle at generation $n-1$. Define the event
$$T_n:=\left\{\forall v\succ w, v\in Z_{n},~\text{we have}~X_v>-\varepsilon n\right\}.$$
Let~$M_{n}$ be the rightmost position of the BRW at time~$n$. By~\cite[(12)]{Gantert18}, for $y>x^*$,
$$\lim_{n\to\infty}\frac{1}{n}\log\mathbb{P}(M_n\geq yn)=-(I(y)-\log m).$$
Since~$\mathbb{P}(k_w>y)=\Theta(1)y^{-\beta}$, we have for $n$ large enough
\begin{align}
&\mathbb{P}(Z_n[xn,\infty)\geq e^{an})\cr
&\geq\mathbb{P}\left(M_{n-1}\geq(x+\varepsilon) n, e^{2an}>k_w>e^{an}, T_n \right)\cr
&\geq\mathbb{P}\left(M_{n-1}\geq(x+\varepsilon) n\right)\mathbb{P}(e^{2an}>k_w>e^{an})\mathbb{P}(X>-\varepsilon n)^{e^{2an}}\cr
&\geq e^{-(I(x+2\varepsilon)-\log m)n} e^{-\beta an}.\nonumber
\end{align}
Taking limits yields that
$$
\liminf_{n\rightarrow\infty}\frac{1}{n}\log\mathbb{P}(Z_n[xn,\infty)\geq e^{an})\geq-\left[(\beta-1)a+a-(\log m-I(x+2\varepsilon))\right].
$$
The desired lower bound follows by letting~$\varepsilon\to0+$.
\par
\textbf{Upper bound.}
Fix~$\varepsilon\in(0,x/2)$. By Markov inequality, for $n$ large enough,
\begin{align}\label{4rrfgt1}
\mathbb{P}(Z_n[xn,\infty)\geq e^{an})&\leq\mathbb{P}\left(\sum_{u\in Z_{n-1}\atop S_u\geq n(x-\varepsilon)}k_u\geq e^{an}\right)+\mathbb{P}\left(\sum_{u\in Z_{n-1}\atop S_u\geq n(x-\varepsilon)}k_u<e^{an},Z_n[xn,\infty)\geq e^{an}\right)\cr
&\leq\mathbb{P}\left(\sum_{u\in Z_{n-1}\atop S_u\geq n(x-\varepsilon)}k_u\geq e^{an}\right)+\mathbb{P}\left(\exists u\in Z_n, X_u\geq\varepsilon n\right)\cr
&\leq \mathbb{E}\left[e^{-an}\sum_{u\in Z_{n-1}\atop S_u\geq n(x-\varepsilon)}k_u\right]+m^n\mathbb{P}\left(X\geq \varepsilon n\right)\cr
&\leq e^{-an}m^{n}\mathbb{P}(S_{n-1}\geq n(x-\varepsilon))+m^ne^{-\frac{\varepsilon^2n^2}{2}}\cr
&\leq 2e^{-an}m^{n}e^{-I(x-2\varepsilon)n},
\end{align}
where the last inequality follows from the Cram\'er theorem (see \cite[Sec 2.2]{Dembo}). Thus,
$$
\limsup_{n\rightarrow\infty}\frac{1}{n}\log\mathbb{P}(Z_n[xn,\infty)\geq e^{an})\leq-\left[a-(\log m-I(x-2\varepsilon))\right].
$$
 The desired upper bound follows by letting $\varepsilon\rightarrow0+$.
\end{proof}

\noindent\textbf{Acknowledgements}
Shuxiong Zhang is supported in part by Natural Science Foundation of Guangdong Province of China (Grant No. 2214050003543).

\smallskip
\noindent \textbf{Competing Interests}
There were no competing interests to declare which arose during the preparation or publication process of this article.

\smallskip
\noindent \textbf{Availability of data and materials}
Data sharing is not applicable to this article as no datasets were generated or analyzed during the current study.

\vspace{0.2cm}

\noindent Shuxiong Zhang\\
School of Mathematics and Statistics, Anhui Normal University, Wuhu, China.\\
E-mail: shuxiong.zhang@mail.bnu.edu.cn\\

\vspace{0.2cm}

\noindent Lianghui Luo\\
School of Mathematical Sciences, Beijing Normal University, Beijing, China.\\
E-mail: lianghui.luo@qq.com\\

\end{document}